\documentclass [a4paper,12pt]{amsart}
\usepackage{amsthm,amstext,amsfonts,lineno}
\usepackage{amsmath}
\usepackage{amssymb}
\usepackage[ansinew]{inputenc}
\usepackage{amscd}
\usepackage[ansinew]{inputenc}
\usepackage[mathscr]{euscript}
\usepackage{color}

\usepackage[all]{xy}
\usepackage{enumitem}

\usepackage{hyperref}

\newtheorem{thm}{Theorem}[section]
\newtheorem{cor}[thm]{Corollary}
\newtheorem{lem}[thm]{Lemma}
\newtheorem{prop}[thm]{Proposition}

\theoremstyle{definition}
\newtheorem{ex}[thm]{Example}

\theoremstyle{definition}
\newtheorem{defn}[thm]{Definition}

\theoremstyle{definition}
\newtheorem{rem}[thm]{Remark}

\theoremstyle{definition}

\def\C{\mathbb C}

\def\Z{\mathbb Z}

\def\T{\textnormal{\hspace{-0.00cm}\textsf T}}
\def\icis{\textsc{icis}}
\def\dim{\operatorname{dim}}
\def\GL{\operatorname{GL}}

\def\Der{\operatorname{Der}}
\def\Derlog{\operatorname{Derlog}}

\def\t{t}

\def\syz{\operatorname{syz}}
\def\Tor{\operatorname{Tor}}

\def\ord {\operatorname{ord}}

\def\ovarphi {\overline\varphi}
\def\opsi {\overline\psi}
\def\O{\mathcal O}

\def\m{\mathbf m}

\def\I{\mathbf I}

\def\J{\mathbf J}

\def\lto{\longrightarrow}

\def\({\left(}
\def\){\right)}

\def\geq{\geqslant}
\def\leq{\leqslant}

\def\*{\color{red}\blacksquare}

\def\+{\color{green}\blacksquare}

\newcommand{\smallfrac}[2]{{\textnormal{\small$\frac{#1}{#2}$}}}

\subjclass[$2010$ Mathematics Subject Classification]{Primary
32S05; Secondary 32S50, 32S99}

\makeatletter
\newcommand{\tpitchfork}{%
  \vbox{
    \baselineskip\z@skip
    \lineskip-.52ex
    \lineskiplimit\maxdimen
    \m@th
    \ialign{##\crcr\hidewidth\smash{$-$}\hidewidth\crcr$\pitchfork$\crcr}
  }%
}

\makeatletter
\def\subsection{\@startsection{subsection}{2}%
  \z@{.5\linespacing\@plus.7\linespacing}{.3\linespacing}%
  {\normalfont\bfseries}}
\makeatother

\begin{document}

\title[Modules of derivations and singularities of pairs of maps]{
Modules of derivations, logarithmic ideals \\ \vskip4pt  and singularities of maps on analytic varieties}

\author{C. Bivi\`a-Ausina}
\address{
Institut Universitari de Matem\`atica Pura i Aplicada,
Universitat Polit\`ecnica de Val\`encia,
Cam\'i de Vera, s/n,
46022 Val\`encia, Spain}
\email{carbivia@mat.upv.es}

\author{K. Kourliouros}
\address{
Imperial College London,
Department of Mathematics,
180 Queen's gate, South Kensington Campus,
London SW7 2AZ, United Kingdom}
\email{k.kourliouros@gmail.com}

\author{M. A. S. Ruas}
\address{
Instituto de Ciências Matemáticas e de Computação,
Universidade de São Paulo,
Av. Trabalhador São-carlense, 400,
13566-590 São Carlos, SP, Brazil}
\email{maasruas@icmc.usp.br}

\keywords{Milnor number, logarithmic vector fields, Tjurina number, parameter submodules, Buchsbaum-Rim multiplicity}

\begin{abstract}
We introduce the module of derivations $\Theta_{h,M}$ attached to a given analytic map $h:(\C^n,0)\to (\C^p,0)$
and a submodule $M\subseteq \O_n^p$ and analyse several exact sequences related to $\Theta_{h,M}$. Moreover, we obtain formulas for several numerical invariants associated to the pair $(h,M)$ and a given analytic map germ $f:(\C^n,0)\to (\C^q,0)$.
In particular, if $X$ is an analytic subvariety of $\C^n$, we derive expressions
for analytic invariants defined in terms of the module $\Theta_X$ of logarithmic vector fields of $X$.
\end{abstract}

\maketitle

\section{Introduction}

Let $X$ be the germ at $0$ of an analytic subvariety of $\C^n$ and let $\O_n$ denote the ring of analytic function germs $f:(\C^n,0)\to \C$. We denote by $\Theta_X$ the $\O_n$-module of germs at $0$ of vector fields of $\C^n$ which are tangential to $X$. That is, if $I(X)$ denotes the
ideal of $\O_n$ generated by the elements of $\O_n$ vanishing on $X$, then
$\Theta_X=\{\delta\in \Der(\O_n): \delta(I(X))\subseteq I(X)\}$, where $\Der(\O_n)$ is the $\O_n$-module of derivations of $\O_n$.
The module $\Theta_X$, which is also usually denoted by $\Derlog(X)$, determines and is determined by a natural stratification on $X$ called
the {\it logarithmic stratification} of $X$ (see \cite[Lemma 1.5]{BR} and \cite[\S3]{Saitolog}). Therefore $\Theta_X$ links the topology of $X$ with the algebraic properties of the ideal $I(X)$, and encodes fundamental properties of $X$. For instance, it is proven in \cite{MirandaN2018} that
the module $\Theta_X$ is reflexive if and only if $X$ is the germ of a hypersurface.

Given a function germ $f\in \O_n$, the ideal of $\O_n$ given by all elements $\delta(f)$, where $\delta$ varies in $\Theta_X$, encodes essential information concerning the singularities of $f$ on any logarithmic stratum of $X$ (see \cite{BR}).
We usually denote this ideal by $J_X(f)$ and we refer to the ideals thus defined as {\it logarithmic ideals}. The colength of $J_X(f)$ in $\O_n$, which we denote by $\mu_X(f)$, constitutes an essential numerical invariant of $f$ under $\mathcal R_X$-equivalence, that is, under biholomorphisms $\phi:(\C^n,0)\to (\C^n,0)$ such that $\phi(X)=X$ (see \cite{BR}).
 We remark that, if $f$ has an isolated singularity at the origin, then we have $\mu_X(f)=\mu(f)$ when considering $X=\C^n$, where
$\mu(f)$ denotes the Milnor number of $f$. Let us recall that if $g:(\C^n,0)\to (\C^p,0)$ denotes an {\icis}, i.\,e. an isolated complete intersection singularity, then we denote the Milnor and the Tjurina number of $g$ by $\mu(g)$ and by $\tau(g)$, respectively (see for instance \cite{GLS, Greuel, Le, VosegaardMathAnn}).

Logarithmic ideals have been an object of study in \cite{A1,BKR1,BiviaRuas,BR,ES,Kour,LNOT3,LNOT4,LNOT5,Ly}, among others.
One of the fundamental properties of these ideals is that, when $f,h\in \O_n$ have isolated singularity at the origin and $\mu_X(f)$ is finite, where $X=h^{-1}(0)$, then $(f,h)$ is an {\icis} and $\mu_X(f)$ can be computed by means of the formula
$$
\mu_X(f)=\mu(f)+\mu(f,h)+\mu(h)-\tau(h)
$$
as proven independently in \cite{Kour} and \cite{LNOT3}. This formula has been generalized in \cite[p.\,45]{LNOT5} in the case where $X$ is replaced
by an {\icis}; the resulting expression has been part of our motivations in this paper.

As we will see, logarithmic ideals take part of a more general setup. One of the purposes of this paper is to study in a unified way several numerical invariants associated to pairs of maps $h:(\C^n,0)\rightarrow (\C^p,0)$ and $f:(\C^n,0)\rightarrow (\C^q,0)$, or equivalently to diagrams of the form
\begin{equation}\label{DD1}
(\C^p,0)\stackrel{h}{\longleftarrow}(\C^n,0)\stackrel{f}{\longrightarrow}(\C^q,0)
\end{equation}
under the action of various subgroups of the contact group $\mathcal{K}$,
naturally acting on such diagrams (see \cite{DamonMemoirsAMS2,Dufour1977,MRT}).
Any diagram of holomorphic map germs like (\ref{DD1}) is usually called a divergent diagram (see for instance \cite{Dufour1977,Dufour1989}). Let us remark that, after identifying (\ref{DD1})
with the map $(f,h):(\C^n,0)\to (\C^{q+p},0)$, the equivalence of divergent diagrams under the action
of a given subgroup of $\mathcal{K}$ in the target of $(f,h)$  corresponds to the product of the corresponding actions in $(\C^p,0)$ and $(\C^q,0)$. Below we describe in more detail both the goals as well as the structure of the paper.

Let $h:(\C^n,0)\rightarrow (\C^p,0)$ be an analytic map germ and let $M$ be a submodule of $\O_n^p$.
In Section \ref{exactseqs} we generalize the definition of $\Theta_X$ by introducing the module $\Theta_{h,M}$ of derivations associated to the pair formed by $h$ and $M$ (see Definition \ref{ThetaC}).
Considering the natural componentwise action of any derivation of $\O_n$ on germs of analytic maps, $\Theta_{h,M}$ is defined as the submodule of $\Der(\O_n)$ formed by the derivations $\delta$ for which $\delta(h)\in M$. Given now a map germ $f:(\C^n,0)\rightarrow (\C^q,0)$, we can look at the submodule $J_{h,M}(f)$ of $\O_n^q$ generated by all elements $\delta(f)$, where $\delta$ varies in $\Theta_{h,M}$.

Let $D(h)$ denote the submodule of $\O_n^p$ generated by $\frac{\partial h}{\partial x_1},\dots, \frac{\partial h}{\partial x_n}$. We say that the module $M$ is {\it geometric} for $h$ if the submodule $D(h)+M\subseteq \O_n^p$ represents the tangent space to the orbit of $h$ under the action of some geometric subgroup $\mathcal{G}$ of the contact group $\mathcal{K}$ in the sense of J.\,Damon \cite{DamonMemoirsAMS2,DamonMemoirsAMS3}.
Let $M\subseteq \O_n^p$ and $N\subseteq \O_n^p$ be geometric submodules for the analytic maps
$h:(\C^n,0)\rightarrow (\C^p,0)$ and $f:(\C^n,0)\rightarrow (\C^q,0)$, respectively, and let $\mathcal G$ and $\mathcal G'$ be the respective associated geometric subgroups leading to the corresponding tangent spaces. The submodule $J_{h,M}(f)+N\subseteq \O_n^q$ gives the tangent space to the $\mathcal G_0'$-orbit of $f$, where $\mathcal G_0'$ is the subgroup of $\mathcal{G}'$ consisting of all diffeomorphisms in $\mathcal{G}'$ which  preserve a fixed representative of the $\mathcal{G}$-orbit of $h$.
In Theorem \ref{claudetot} we show an abstract version of the fundamental short exact sequence of infinitesimal deformations modules associated to diagrams of maps $(f,h)$ for arbitrary actions of subgroups of the group $\mathcal{K}$ naturally acting on such diagrams. Versions of this sequence appear either implicitly or explicitly in several works (see for instance \cite{HRR} or \cite{MRT}). Using this, in Theorem \ref{claudetot}\,(\ref{ses2}) and Proposition \ref{prop-1}\,(\ref{es2KM}) we show formulas for the colength of modules of the form $J_{h,M}(f)+N$, where $N\subseteq\O_n^q$ is any submodule of $\O_n^q$. As a result, we obtain several expressions concerning the ideals $J_X(f)$ in the case $q=1$ and specially when $X$ is an isolated complete intersection singularity (see Corollary \ref{segonaexp}) or a union of hypersurfaces (see Corollary \ref{unionXi}).

Let $I$ be an ideal of $\O_n$. In Section \ref{RSQM} we focus on analysing how Theorem \ref{claudetot}\,(\ref{ses2}) can be improved to obtain an expression for the colength of $J_{h,M}(f)+N$ in $\O_n^q$ when $M=I^{\oplus p}$ and $N$ is a submodule of $I^{\oplus q}$,
where $f:(\C^n,0)\to (\C^q,0)$ and $h:(\C^n,0)\to (\C^p,0)$ denote analytic maps. In this sense, the main result of this section is Corollary \ref{expgen1}, where we obtain an expression of the colength of a submodule of this type. This result is preceded by several algebraic results that has lead us to highlight the condition $D(g)\cap I^{\oplus s}=ID(g)$ (see Propositions \ref{propTor} and \ref{consistent}), for a given analytic map
$g:(\C^n,0)\to \C^s$. When this condition holds then we say that $g$ is {\it $I$-consistent}. Thus, in Corollary \ref{expgen1} we prove that if
$D(f,h)+N\oplus I^{\oplus p}$ has finite colength and $(f,h)$ is $I$-consistent, then
$$
\dim_\C\frac{\O_n^q}{J_{h, I^{\oplus p}}(f)+N}= \dim_\C\frac{I^{\oplus q}}{ID(f)+N}+\dim_\C\frac{\O^{q+p}_n}{D(f,h)+I^{\oplus (p+1)}}-
\dim_\C\frac{\O^p_n}{D(h)+I^{\oplus p}}.
$$
Also in Section \ref{RSQM} we characterize
the notion of $I$-consistency of analytic maps $g:(\C^n,0)\to \C^s$ in several cases, especially when the quotient module $\O_n^s/D(g)$ is Cohen-Macaulay (see Theorem \ref{JXfIf}).

In Section \ref{logideals} we apply the results of Section \ref{RSQM} to deduce a formula for the colength of ideals of the form
$J_X(f)+N$, where now $X=h^{-1}(0)$ and $h:(\C^n,0)\to (\C^p,0)$ is an {\icis}, $f:(\C^n,0)\to (\C,0)$ is an analytic function germ such that $\mu_X(f)$ is finite and $N$ is any ideal of $\O_n$ contained in the ideal generated by the component functions of $h$ (see Corollary \ref{noumuXf}). The case $N=0$ of this expression coincides with the mentioned formula obtained in \cite[p.\,45]{LNOT5} by following a different procedure.

Finally, motivated by the known result expressing the Tjurina number of an {\icis} $h:(\C^n,0)\to (\C^p,0)$ as the dimension of a quotient of $\Theta_X$ (see \cite[3.6]{Vosegaard}, \cite[9.6]{Looijenga} or \cite[Proposition 4.6]{LNOT5}), our goal in Section \ref{ThetaXT} is to extend this result to pairs of maps $f:(\C^n,0)\to (\C^q,0)$ and $h:(\C^n,0)\to (\C^p,0)$ for which $(f,h)$ is $I$-consistent, where $I$ denotes an ideal of $\O_n$ (see Proposition \ref{lesTheta2} and Corollary \ref{ThThT}). The case of our results where $h$ is an {\icis}, $f$ is a function and $I$ is the ideal generated by the components of $h$ leads to the existing results in this direction, as can be seen in Corollary \ref{isomicis}.

\section{Exact sequences associated to modules of derivations acting on maps}\label{exactseqs}

Let $h:(\C^n,0)\to (\C^p,0)$ and $f:(\C^n,0)\to (\C^q,0)$ be germs of analytic maps and let $M$ and $N$ denote
submodules of $\O_n^p$ and $\O_n^q$, respectively. In this section we introduce an exact sequence associated to the pairs $(h,M)$ and $(f,N)$.
This exact sequence (see Theorem \ref{claudetot}) has lead us to show a fundamental relation between the respective colengths of some submodules attached to these pairs, as can be seen in Theorem \ref{claudetot}. Moreover, based on another exact sequence, in Proposition \ref{prop-1} and Corollary \ref{codsumats} we expose additional relations between numerical invariants, defined in terms of modules of derivations, attached to such pairs.

 Let $h=(h_1,\dots, h_p):(\C^n,0)\to (\C^p,0)$ be an analytic map
germ. We denote by $D(h)$ the submodule of $\O_n^p$ generated by
$\frac{\partial h}{\partial x_1},\dots, \frac{\partial h}{\partial x_n}$. We refer to $D(h)$ as the {\it Jacobian module of $h$}. We implicitly consider each partial derivative
$\frac{\partial h}{\partial x_i}$ as a column matrix, so we will also regard $D(h)$ as a $p\times n$ matrix. If $f\in \O_n$, we denote $D(f)$ also by $J(f)$ and we refer to it as the {\it Jacobian ideal of $f$}.

Let $R$ be a ring. If $J$ denotes any ideal of $R$ and $p\in\Z_{\geq 1}$, then we denote by $J^{\oplus p}$ the submodule
$J\oplus \cdots \oplus J\subseteq R^p$, where $J$ is repeated $p$ times.

\begin{defn}\label{ThetaC}
Let $h=(h_1,\dots, h_p):(\C^n,0)\to (\C^p,0)$ be an analytic map germ and let $M\subseteq \O_n^p$ be a submodule.
\begin{enumerate}[label=(\alph*)]
\item We usually identify the module of derivations $\Der(\O_n)$ of $\O_n$ with $\O^n_n$. Given a derivation $\delta\in \O_n^n$,
we denote by $\delta(h)$ the column vector whose element in the $i$-row is equal to $\delta(h_i)$, for all $i=1,\dots, p$. Hence,
if $\delta=(\delta_1,\dots, \delta_n)\in \O_n^n$, then $\delta(h)=\delta_1\frac{\partial h}{\partial x_1}+\cdots+
\delta_n\frac{\partial h}{\partial x_n}$.

\item\label{apb} We denote by $\Theta_{h,M}$ the submodule of $\O_n^n$
generated by those $\delta\in\O_n^n$ for which $\delta(h)\in M$. Obviously $\Theta_{h,M}=\Theta_{h,M\cap D(h)}$. If $N$ is another submodule of $\O_n^p$, then we
define $\Theta_{N,M}$ as the submodule of $\O_n^n$
generated by those $\delta\in\O_n^n$ such that $\delta(N)\subseteq M$, where
$\delta(N)$ denotes the submodule of $\O_n^p$ generated by $\delta(g)$, where $g$ varies in $N$. If $g_1,\dots, g_r$ is a generating system of $N$, then
$\Theta_{N,M}=\Theta_{g_1,M}\cap \cdots \cap\Theta_{g_r,M}$ provided that $N\subseteq M$, but
it is easy to find examples where this equality is not true in general.
We denote the submodule $\Theta_{M,M}\subseteq \O_n^n$ simply by $\Theta_M$.

\item If $f:(\C^n,0)\to (\C^q,0)$ is another analytic map, we
denote by $J_{h,M}(f)$ the submodule of $\O_n^q$ generated by $\{\delta(f): \delta\in \Theta_{h,M}\}$.

\item\label{notacio} In order to simplify the notation, if $I$ denotes any ideal of $\O_n$, then
we denote the modules $\Theta_{h,I^{\oplus p}}$ and $J_{h,I^{\oplus p}}(f)$ simply by
$\Theta_{h, I}$ and $J_{h, I}(f)$, respectively. We denote by $H_h$
the module $\Theta_{h, 0}$ (where here $0$ denotes the zero ideal of $\O_n$) and we refer to it as the module of {\it Hamiltonian vector fields associated to $h$}.
\end{enumerate}
\end{defn}

\begin{rem}
Let $I$ be an ideal of $\O_n$ and let $h_1,\dots, h_p$ be a generating system of $I$.
According to Definition \ref{ThetaC}, the module $\Theta_I$ of derivations $\delta\in\O_n^n$ for which $\delta(I)\subseteq I\}$
is equal to $\Theta_{h, I}$, where $h$ denotes the map $(h_1,\dots, h_p)$.
Given a derivation $\delta\in \O_n^n$, any ideal of $\O_n$ for which $\delta(I)\subseteq I$ is called
{\it $\delta$-invariant} in \cite{Ciuperca}. We recall that Ciuperc\u{a} proves in \cite[Theorem 2.4]{Ciuperca}, in a more general context, that any derivation
$\delta\in \O_n^n$ for which $\delta(I)\subseteq I$ also verifies that $\delta(\overline I)\subseteq \overline I$, where
$\overline I$ denotes the integral closure of $I$. Moreover, the module
$\Theta_I$ is also called the {\it tangential idealizer} of $I$ in \cite{MirandaN2016}. We remark that, if $I$ denotes an ideal of $\O_n$, the problem of determining an explicit
generating system of $\Theta_I$ is an active problem in commutative algebra (see for instance \cite{BM-N}) and that
$\Theta_I$ also encodes a lot of information regarding the geometry of the zero set of $I$, as can be seen in \cite{BR,DamonMemoirsAMS,HM,MirandaN2018,Saitolog}.
\end{rem}

Let $R$ be a ring and let $M$ be an $R$-module. Given a sequene of elements $u_1,\dots, u_s\in M$, we denote by
$\syz(u_1,\dots, u_s)$ the module of syzygies of $u_1,\dots, u_s$. That is, $\syz(u_1,\dots, u_s)$ is the $R$-submodule
of $R^s$ generated by those $g_1,\dots, g_s$ for which $g_1u_1+\cdots +g_su_s=0$. In particular, we have
$H_h=\syz(\frac{\partial h}{\partial x_1},\dots,\frac{\partial h}{\partial x_n})$, for any analytic map germ $h:(\C^n,0)\to (\C^p,0)$.

The next result is analogous to \cite[Lemma 2.1]{BR} and allows us to apply Singular \cite{Singular} to obtain a matrix of generators
of modules of the form $\Theta_{h,M}$.

\begin{lem}\label{syzyg}
Let $h=(h_1,\dots, h_p):(\C^n,0)\to (\C^p,0)$ be an analytic map and let $M$ be any submodule of $\O_n^p$.
Let us consider a matrix $[m_{ij}]$ of size $p\times r$, for some $r\geq 1$, whose columns generate $M$.
Let $D$ be the set of elements of $\O_n^p$ given by the columns of the matrix
\begin{equation}\label{matrixDerlog}
\left[
  \begin{array}{ccccccccccccc}
\frac{\partial h_1}{\partial x_1}&\cdots & \frac{\partial h_1}{\partial x_n}& m_{11} & \cdots & m_{1r} \\
\vdots &  & \vdots & \vdots &  & \vdots  \\
\frac{\partial h_p}{\partial x_1}&\cdots & \frac{\partial h_p}{\partial x_n}& m_{p1} & \cdots & m_{pr}
  \end{array}
\right].
\end{equation}
Then $\Theta_{h, M}=\pi_n(\syz(D))$, where $\pi_n:\O_n^{n+r}\to \O_n^n$ is the projection onto the first $n$ components.
\end{lem}

\begin{proof}
Let $\delta\in\O_n^n$. We have that
$\delta=(\delta_1,\dots,\delta_n)$ belongs to $\Theta_{h, M}$ if and only if
there exist $a_1,\dots, a_m\in\O_n$ such that
$\delta_1\frac{\partial h_i}{\partial x_1}+\cdots +\delta_n\frac{\partial h_i}{\partial x_n}=
a_1m_{i1}+\cdots+a_rm_{ir}$, for all $i=1,\dots, p$. This latter condition is equivalent to saying that the element of $\O_n^{n+r}$ given by
$(\delta_1,\dots, \delta_n, -a_1,\dots, -a_r)$ belongs to $\syz(D)$. Hence the result follows.
\end{proof}

Let $U$ be an open neighbourhood of $0\in\C^n$ and let $\mathcal N\subseteq\mathcal M$ denote coherent sheaves of submodules of $\O_U^p$
defined in $U$. Therefore, in Definition \ref{ThetaC} we introduce a new sheaf $\Theta_{\mathcal N,\mathcal M}$ defined also in $U$ whose stalks
are given by $(\Theta_{\mathcal N,\mathcal M})_z=\Theta_{\mathcal N_z,\mathcal M_z}$, for all $z\in U$. This sheaf of modules is coherent, as a direct application of Oka's theorem (see \cite[p.\,320]{L}) and Lemma \ref{syzyg}.

Let $h:(\C^n,0)\to (\C^p,0)$ be an analytic map germ.
We denote by $T(h)$ the submodule of $\O_n^p$ given by
$D(h)+\langle h_1,\dots, h_p\rangle\O_n^p$.
That is, $T(h)$ is generated by the columns of the following matrix:
\begin{equation}\label{matrixTh}
T(h)=\left[
  \begin{array}{ccccccccccccc}
\frac{\partial h_1}{\partial x_1}&\cdots & \frac{\partial h_1}{\partial x_n}& h_1 & \cdots & h_p & 0 & \cdots &0 & \cdots & 0 &\cdots & 0 \\
\frac{\partial h_2}{\partial x_1}&\cdots & \frac{\partial h_2}{\partial x_n}& 0 & \cdots & 0 & h_1 & \cdots & h_p &  & 0 &\cdots & 0 \\
\vdots &  & \vdots & \vdots &  & \vdots & \vdots &  & \vdots & \ddots & \vdots  & & \vdots \\
\frac{\partial h_p}{\partial x_1}&\cdots & \frac{\partial h_p}{\partial x_n}& 0 & \cdots & 0 & 0 & \cdots &0 &  & h_1 & \cdots & h_p
  \end{array}
\right].
\end{equation}
It is well known that this module is the tangent space to the orbit of $h$ under the action of Mather's contact group $\mathcal K$
(see for instance \cite{Wall.ProcLMS}).

Let $h=(h_1,\dots, h_p):(\C^n,0)\to (\C^p,0)$ be an analytic map germ, $p\leq n$. We denote by $\J(h_1,\dots,h_p)$
the ideal of $\O_n$ generated by the maximal minors of the differential matrix $D(h_1,\dots, h_p)$. We recall that $h$ is called an
{\it isolated complete intersection singularity}, or an {\icis}, for short, when the dimension of $h^{-1}(0)$ equals $n-p$ and
the ideal $\langle h_1,\dots, h_p\rangle +\J(h_1,\dots, h_p)$ has finite colength in $\O_n$. In this case, the zero set
 $h^{-1}(0)$ is also called an {\icis}. We recall that, if $h$ is an {\icis} with $n-p\geq 1$, then the ideal of $\O_n$ generated by
 the component functions of $h$ is radical (see \cite[p.\,7]{Looijenga}).

Let $h:(\C^n,0)\to (\C^p,0)$ be an {\icis}. We denote by $\mu(h)$ the Milnor number of $h$ (see for instance
\cite{Greuel,Le,Looijenga}). The {\it Tjurina number} of $h$ is defined as follows (see \cite[p.\,246]{GLS}):
\begin{equation}\label{tauicis}
\tau(h)=\dim_\C\frac{\O_n^p}{T(h)}.
\end{equation}
Usually the Milnor number of $h$ is computed by means of the Lê-Greuel formula
(see Greuel \cite[p.\,263]{Greuel} and Lê \cite[p.\,130]{Le}). As a consequence, if the map $(h_1,\dots, h_r)$ is an
 \textsc{icis}, for all $r\in\{1,\dots,
p\}$, then the Milnor number of $h$ is given by the following
formula:
\begin{equation}\label{LeGr}
\mu(h)=\sum_{r=1}^p (-1)^{p-r}\dim_\C\frac{\O_n}{\langle h_1,\dots,
h_{r-1}\rangle +  \mathbf J(h_1,\dots, h_r)},
\end{equation}
where $h_0$ is defined as the zero function.

In \cite{LS} it is proven that if $h:(\C^n,0)\to (\C^p,0)$ is an {\icis} such that $n-p\geq 1$, then $\mu(h)\geq \tau(h)$
and both numbers coincide provided that $h$ is weighted homogeneous (see \cite{GreuelDualitat, Vosegaard}).
In \cite{Al} the case where $h$ is zero-dimensional is thoroughly analysed, in particular it is proven that $\mu(h)\leq \tau(h)$ in this case. Let us remark that even if $h$ is an {\icis} whose component functions are
homogeneous polynomials of the same degrees, if $h^{-1}(0)$ is zero-dimensional, then $\mu(h)\neq\tau(h)$, in general. For instance, if $h:(\C^2,0)\to (\C^2,0)$ is given by $h(x,y)=(x^3+y^3, x^2y)$, then $\mu(h)=8$ and $\tau(h)=11$.

\begin{rem}\label{introdThetaX}   
\begin{enumerate}[label=(\alph*),wide]
\item\label{adeintrodThetaX} Let $X$ be an analytic subvariety of $(\C^n,0)$ and let $I=I(X)$, that is, $I$ is the ideal of $\O_n$ formed by those function germs vanishing on $X$. Let $\Theta_X$ be the module of those derivations $\delta\in \Der(\O_n)$
for which $\delta(I)\subseteq I$. This module is usually known as the module of {\it logarithmic derivations of $X$} and is also denoted by $\Derlog(X)$ (see for instance \cite{BR, DamonMemoirsAMS}). If $h:(\C^n,0)\to (\C^p,0)$ is an analytic map whose component functions generate $I$,
we have $\Theta_X=\Theta_{h,I}$, according with the notation introduced in Definition \ref{ThetaC}\,\ref{notacio}.
In order to obtain $\Theta_X$, we can apply Lemma \ref{syzyg}
by replacing the matrix in (\ref{matrixDerlog}) by the matrix $T(h)$ defined in (\ref{matrixTh}).
\item\label{defmuX} Keeping the notation used above, given any function germ $f\in\O_n$, we denote the ideal $J_{h,I(X)}(f)$ simply by $J_X(f)$. Let us recall that the colength of the ideal $J_X(f)$ in $\O_n$, when finite, is usually called the {\it Bruce-Roberts number of $f$ with respect to $X$}
(see for instance \cite{BiviaRuas, Kour, LNOT4, LNOT5}) and we will denote it by $\mu_X(f)$.
This number generalizes the Milnor number
of a function (by considering the case $X=(\C^n,0)$) and has been extensively studied in \cite{BR}.
We will refer to the ideals of the form $J_X(f)$, for some $f\in\O_n$ and some analytic subvariety $X$ of $(\C^n,0)$
as {\it logarithmic ideals}.
\end{enumerate}
\end{rem}

The following result is motivated mainly by \cite[Proposition 1]{HRR} and \cite[Proposition 2.1]{MRT}.
As we will see in subsequent sections, this result will lead us to derive fundamental formulas connecting several numerical invariants.

\begin{thm}\label{claudetot} Let $h:(\C^n,0)\to (\C^p,0)$ and $f:(\C^n,0)\to (\C^q,0)$ be analytic map germs and let $M\subseteq \O_n^p$
and $N\subseteq\O_n^q$ be submodules.
Then the following short exact sequence is exact:
\begin{equation}\label{ses}
\xymatrix@C=0.2cm@R=1.5ex{
0 \ar[rr] &&  \displaystyle \frac{\O_n^q}{ J_{h,M}(f)+N} \ar[rr]^-{\ovarphi} &&
\displaystyle\frac{\O^{q+p}_n}{D(f,h)+(N\oplus M)}  \ar[rr]^-{\opsi}  && \displaystyle    \frac{\O^p_n}{D(h)+M}  \ar[rr] && 0
}
\end{equation}
where $\ovarphi$ and $\opsi$ are the morphisms induced by the inclusion $g\mapsto (g,0)$, for all $g\in \O_n^q$, and the projection
onto the last $p$ components, respectively. In particular,
the submodule $D(f,h)+(N\oplus M)$ of $\O_n^{q+p}$ has finite colength if and only if
the submodules $J_{h,M}(f)+N$ and $D(h)+M$ have finite colength. In this case, we obtain the following relation
\begin{equation}\label{ses2}
\dim_\C\frac{\O_n^q}{ J_{h,M}(f)+N}=\dim_\C\frac{\O^{q+p}_n}{D(f,h)+(N\oplus M)} -\dim_\C\frac{\O^p_n}{D(h)+M}.
\end{equation}
\end{thm}

\begin{proof}
Let $\varphi:\O_n^q\to \O_n^{q+p}$ denote the inclusion $g\mapsto (g,0)$, for all $g\in \O_n^q$, and let
$\psi:\O_n^{q+p}\to \O_n^p$ be the projection onto the last $p$ components. Clearly we have $\psi\circ\varphi=0$.

Let $g\in  J_{h,M}(f)+N$. Then $g=\delta(f)+a$, for some $\delta\in \Theta_{h,M}$ and some $a \in N$.
By the definition of $\Theta_{h,M}$, there exists an element $m=(m_1,\dots, m_p)\in M$ for which
$\delta(h)+m=0$. Let us write $h=(h_1,\dots, h_p)$.
Thus $(g,0,\dots, 0)=(\delta(f),\delta(h_1),\dots,\delta(h_p))+(a,m_1,\dots, m_p)\in D(f,h)+(N\oplus M)$.
That is $\varphi(J_{h,M}(f)+N)\subseteq D(f,h)+(N\oplus M)$.
Obviously $\psi(D(f,h)+(N\oplus M))=D(h)+M$. Thus the induced
morphisms $\ovarphi$ and $\opsi$ are well defined.

Given any element $g\in \O_n^q$, the condition $(g,0,\dots, 0)\in D(f,h)+(N\oplus M)$ is equivalent to saying that
there exists some derivation $\delta \in \O_n^n$ and some $(a,b)\in N\oplus M$ such that $g=\delta(f)+a$ and $0=\delta(h)+b$,
which in turn is equivalent to the condition $g\in J_{h,M}(f)+N$.
Hence the morphism $\ovarphi$ is injective. Clearly $\opsi$ is surjective.

Let us see the exactness of (\ref{ses}) at the middle module of (\ref{ses}).
Let $g\in \O_n^q$ and let $g' \in D(h)+M\subseteq \O_n^p$.
Then there exist a derivation $\delta\in \O_n^n$ and
an element $m=(m_1,\dots, m_p)\in M$ such that $g'=\delta(h)+m$.
Therefore we have that
\begin{align*}
(g,g')-(g-\delta(f),0)=(\delta(f), \delta(h))+
(0, m)&\in  D(f,h)+(\{0\}\oplus M)\\
&\subseteq D(f,h)+(N\oplus M).
\end{align*}
That is, $\overline\varphi([g-\delta(f)])=
[(g,g')]$, where the brackets denote the corresponding classes modulo $J_{h,M}(f)+N$ and
$D(f,h)+(N\oplus M)$, respectively. This shows that the kernel of $\opsi$ is contained in the image of $\ovarphi$.
Hence (\ref{ses}) is exact.

Relation (\ref{ses2}) follows as an immediate application of the exactness of (\ref{ses}).
\end{proof}

\begin{rem}\label{Diagr}
The short exact sequence (\ref{ses}) above is of significant importance in the deformation theory of divergent diagrams of maps. Indeed, let us denote by $\mathcal K^{(2)}$ the subgroup of the contact group naturally acting on such diagrams. Two diagrams $g=(f,h)$ and $g'=(f',h')$ are $\mathcal K^{(2)}$-equivalent if there exists a diffeomorphism $\phi:(\C^n,0)\to (\C^n,0)$ in the source and a matrix-valued map
$M:\C^n\rightarrow \GL(\C^q)\oplus \GL(\C^p)$, $M=M_1\oplus M_2$, such that $g(\phi(x))=M(x)g'(x)$, or equivalently
\begin{align*}
f(\phi(x))&=M_1(x)f'(x)\\
h(\phi(x))&=M_2(x)h'(x).
\end{align*}

Let $f:(\C^n,0)\to (\C^q,0)$ and $h:(\C^n,0)\to (\C^p,0)$ be analytic map germs. Let us consider now two submodules $N\subseteq\O_n^q$ and $M\subseteq\O_n^p$ which are geometric for $f$ and $h$, with respect to some geometric subgroups $\mathcal{G}'\subseteq \mathcal{K}$ and $\mathcal{G}\subseteq \mathcal{K}$ acting on $\O_n^q$ and $\O_n^p$, respectively. When the module $M\oplus N\subseteq \O_n^{q+p}$ is also geometric for the divergent diagram $g=(f,h),$  the submodule $D(f,h)+(N\oplus M)\subseteq \O_n^{q+p}$ represents the tangent space to the orbit of $(f,h)$ under $(\mathcal{G}',\mathcal{G})$-equivalence, and the quotient module $\frac{\O_n^{q+p}}{D(f,h)+N\oplus M}$ is
	the module of infinitesimal deformations (also called normal space in \cite[\S 3]{Damonsections}) of the $(\mathcal{G}',\mathcal{G})$-equivalence class of the pair $(f,h)$.

The colenght of the module  $D(f,h)+(N\oplus M)$ is an analytic invariant
	of  $(\mathcal{G}',\mathcal{G})$-equivalence, called the  $(\mathcal{G}',\mathcal{G})_{e}$-codimension of $(f,h).$
	Let us denote by $\mathcal{G}'_0$ the subgroup of $\mathcal{G}'$ consisting of diffeomorphisms which preserve a representative of the $\mathcal{G}$-equivalence class of $h$.
Then, the short exact sequence (5)  relates the spaces of infinitesimal deformations of the  $(\mathcal{G}',\mathcal{G})$, $\mathcal{G}'_0$ and $\mathcal{G}$-equivalence classes of $(f,h)$, $f$ and $h$, respectively, providing formulas relating their codimensions.

As we will see below, and it will be analysed in more detail in the next sections, the above short exact sequences relate the basic relative invariants of the map $f$ with respect to $X=h^{-1}(0)$, to the invariants of the diagram $(f,h)$ and of the map $h$ as well, all given by the codimensions of the tangent spaces of their corresponding orbits.
\end{rem}

Let $h=(h_1,\dots,h_p):(\C^n,0)\to (\C^p,0)$ be an analytic map germ and let
$M\subseteq \O_n^p$ be a submodule. Motivated by the definition of Tjurina number of an {\icis} (see (\ref{tauicis})), and in order to simplify the notation, we define the number $\t_M(h)$ given by
\begin{equation}\label{tauM}
\t_M(h)=\dim_\C\frac{\O_n^p}{D(h)+M}
\end{equation}
whenever the module $D(h)+M$ has finite colength in $\O_n^p$. Analogous to
Definition \ref{ThetaC}\,\ref{notacio}, if $I$ is an ideal of $\O_n$ then we will denote $\t_{I^{\oplus p}}(h)$
simply by $\t_{I}(h)$. Obviously, if $h$ is an {\icis} and $I$ denotes the ideal $\langle h_1,\dots,h_p\rangle\subseteq \O_n$,
then $\t_I(h)=\tau(h)$.

\begin{cor}\label{diffsym} Let $h:(\C^n,0)\to (\C^p,0)$ and $f:(\C^n,0)\to (\C^q,0)$ be analytic map germs and let $M\subseteq \O_n^p$
and $N\subseteq\O_n^q$ be submodules. Then the following conditions are equivalent:
\begin{enumerate}[label=\textnormal{(\alph*)}]
\item $\t_{M\oplus N}(h,f)$ is finite.
\item the modules $J_{h,M}(f)+N\subseteq \O_n^q$ and $D(h)+M\subseteq \O_n^p$ have finite colength.
\item the modules $J_{f,N}(h)+M\subseteq \O_n^p$ and $D(f)+N\subseteq \O_n^q$ have finite colength.
\end{enumerate}
Moreover, under any of the above conditions, the following relation holds:
\begin{equation}\label{ses4}
\dim_\C\frac{\O^q_n}{J_{h,M}(f)+N}-\dim_\C\frac{\O^p_n}{J_{f,N}(h)+M}=
\t_N(f)-\t_M(h).
\end{equation}
\end{cor}

\begin{proof}
Let us remark that the submodules $D(f,h)+(N\oplus M)$ and
$D(h,f)+(M\oplus N)$ of $\O_n^{q+p}$ are canonically isomorphic. In particular, by permuting the roles of the
pairs $(h,M)$ and $(f,N)$ in the sequence (\ref{ses}), the exactness of (\ref{ses}) shows the equivalence between (a), (b) and (c).

If (c) holds, by (\ref{ses2}) and the notation introduced in (\ref{tauM}), we obtain that
\begin{equation}\label{ses3}
\dim_\C\frac{\O_n^p}{ J_{f,N}(h)+M}=\t_{M\oplus N}(h,f)-\t_{N}(f).
\end{equation}
Hence, since $\t_{M\oplus N}(h,f)=\t_{N\oplus M}(f,h)$, equality (\ref{ses4}) follows by writing the difference of relations (\ref{ses2}) and (\ref{ses3}).
\end{proof}

If $h:(\C^n,0)\to (\C^p,0)$ is an analytic map germ, we denote by $h^*:\O_p\to \O_n$ the ring morphism obtained by composition with $h$.
Given an integer $r\geq 1$, we will use the same notation for the componentwise extension of this ring morphism leading to a morphism of $\O_p$-modules $\O_p^r\to \O_n^r$ (where we consider the natural structure of $\O_n^r$ as $\O_p$-module via $h^*$).

\begin{ex}
Let $V\subseteq (\C^3,0)$ be the variety given by $V=\varphi^{-1}(0)$, where $\varphi:(\C^3,0)\to (\C,0)$
is the weighted homogeneous polynomial defined by
$$
\varphi(x,y,z)=256z^3-27x^4+144x^2yz-128y^2z^2-4x^2y^3+16y^4z
$$
for all $(x,y,z)\in\C^3$.
This variety, known as the {\it swallowtail} \cite{A2, BR}, is a free divisor. That is, $\Theta_X$ is a free submodule of $\O_3^3$ (see \cite{BR, DamonMemoirsAMS}). Let us consider now the diagram
\begin{equation}\label{DD2}
(\C^3,0)\stackrel{h}{\longleftarrow}(\C^2,0)\stackrel{f}{\longrightarrow}(\C^3,0)
\end{equation}
where $f$ and $h$ are the immersive maps given by $f(u,v)=(u^2,u,v)$ and $h(u,v)=(u,0,v)$, for all $(u,v)\in \C^2$.
Let $M$ and $N$ be the submodules of $\O_2^3$ given by $M=h^*(\Theta_V)$ and $N=f^*(\Theta_V)$.
We notice that the $\mathcal K_{V,e}$-codimensions of $f$ and $h$ are given by
$t_N(f)=2$ and $t_M(h)=1$, respectively (see \cite{DamonAJM,Damonsections}). We can interpret $t_N(f)-t_M(h)$ as a measure that reflects the difference between the contact
with $V$ of the sections $h$ and $f$. By (\ref{ses4}) we also have
$$
\t_N(f)-\t_M(h)=\dim_\C\frac{\O^3_2}{J_{h,M}(f)+N}-\dim_\C\frac{\O^3_2}{J_{f,N}(h)+M}.
$$
Let us remark that, since $t_{M\oplus N}(h,f)=9$, relation (\ref{ses2}) shows that
\begin{equation}\label{doscods}
\dim_\C\frac{\O_2^3}{J_{h,M}(f)+N}=8,\hspace{1cm} \dim_\C\frac{\O_2^3}{J_{f,N}(h)+M}=7.
\end{equation}

By using Singular \cite{Singular} we obtained that the submodules $\Theta_{h,M}$ and $\Theta_{f,N}$ of $\O_2^2$ are generated by the columns of the matrices
$$
\left[
             \begin{array}{cc}
              3u & 64v^2 \\
              4v & 9u^3
             \end{array}
             \right]
\hspace{0.3cm}\textnormal{and}\hspace{0.3cm}
\left[
             \begin{array}{cc}
              8u^3+27u^4 &  16v-20u^2-72u^3 \\
              16v^2+12u^2v+72u^3v+3u^5 & -32uv-192u^2v-7u^4
             \end{array}
             \right]
$$
respectively. By using the above systems of generators we can also obtain (\ref{doscods}) directly, by using the
definitions of the ideals $J_{h,M}(f)$ and $J_{f,N}(h)$.
\end{ex}

The next result relates the colength of the submodule $J_{h,M}(f)+N\subseteq \O_n^q$ with the colength of the submodule $\Theta_{h, M}+\Theta_{f, N}\subseteq \O_n^n$.

\begin{prop}\label{prop-1}
Let $h:(\C^n,0)\to (\C^p,0)$ be an analytic map germ and let $M\subseteq \O_n^p$ be a submodule.
Let $f:(\C^n,0)\to (\C^q,0)$ be an analytic map germ such that the submodule $J_{h, M}(f)+N\subseteq \O_n^q$ has finite colength, for
a given submodule $N\subseteq \O_n^q$.
Then the following relation holds (where all quotient modules involved have finite length):
\begin{equation}\label{es2KM}
\dim_\C\frac{\O^q_n}{J_{h, M}(f)+N}=\t_N(f)+\dim_\C\frac{\O_n^n}{\Theta_{h, M}+\Theta_{f, N}}.
\end{equation}
In particular, when $q=1$, and supposing that the ideal $J_{h,M}(f)$ has finite colength in $\O_n$, we obtain:
\begin{align}
\dim_\C\frac{\O_n}{J_{h, M}(f)}-\dim_\C\frac{\O_n}{J_{h, M}(f)+N}&=\mu(f)-\t_N(f)+\dim_\C\frac{\Theta_{h,M}+\Theta_{f,N}}{\Theta_{h,M}+H_f} \nonumber\\
&=\mu(f)-\t_N(f)+\dim_\C\frac{\Theta_{f,N}}{(\Theta_{f,N}\cap\Theta_{h,M})+H_f}.\label{numinustau2K}
\end{align}
\end{prop}

\begin{proof}
Let $df:\O_n^n\to \O_n^q$ be the evaluation map $\delta\longmapsto \delta(f)$, for any derivation $\delta\in \O_n^n$.
The inclusion $J_{h,M}(f)\subseteq D(f)$ and the morphism $df$ naturally induce the following exact sequences of $\O_n$-modules:
\begin{align}
&\xymatrix@C=0.2cm@R=1.5ex{
0 \ar[rr] &&  \displaystyle \frac{\O_n^n}{\Theta_{h,M}+H_f} \ar[rr]^-{df} &&
\displaystyle\frac{\O_n^q}{J_{h,M}(f)}  \ar[rr]  && \displaystyle\frac{\O_n^q}{D(f)}  \ar[rr] && 0
}\label{ses-1}\\
&\xymatrix@C=0.2cm@R=1.5ex{
0 \ar[rr] &&  \displaystyle \frac{\O_n^n}{\Theta_{h,M}+\Theta_{f, N}} \ar[rr]^-{df} &&
\displaystyle\frac{\O_n^q}{J_{h,M}(f)+N}  \ar[rr]   && \displaystyle\frac{\O_n^q}{D(f)+N}     \ar[rr] && 0
}\label{ses-2}
\end{align}
where the respective third morphisms of (\ref{ses-1}) and (\ref{ses-2}) are the natural projections.
The exactness of (\ref{ses-2}) gives relation (\ref{es2KM}).

In the rest of the proof we will assume that $q=1$ and that $J_{h,M}(f)$ has finite colength.
Let $N$ denote any ideal of $\O_n$. Then (\ref{es2KM}) says that
\begin{equation}\label{es2KM2}
\dim_\C\frac{\O_n}{J_{h, M}(f)+N}=\dim_\C\frac{\O_n}{J(f)+N}+\dim_\C\frac{\O_n^n}{\Theta_{h, M}+\Theta_{f, N}}.
\end{equation}
In particular, we obtain
\begin{equation}\label{es2KM20}
\dim_\C\frac{\O_n}{J_{h, M}(f)}=\mu(f)+\dim_\C\frac{\O_n^n}{\Theta_{h, M}+H_f}.
\end{equation}

The exactness of (\ref{ses-1}) shows that $\Theta_{h,M}+H_f\subseteq \O_n^n$
has finite colength (because we suppose that $J_{h,M}(f)$ has finite colength). As a consequence, since $H_f\subseteq \Theta_{f, N}$, we obtain the relation
\begin{equation}\label{hMfK}
\dim_\C\frac{\Theta_{h,M}+\Theta_{f,N}}{\Theta_{h,M}+H_f}=\dim_\C\frac{\O_n^n}{\Theta_{h,M}+H_f}-\dim_\C\frac{\O_n^n}{\Theta_{h,M}+\Theta_{f, N}},
\end{equation}
where all quotient modules involved in this equality have finite length.

We also have the following isomorphism of $\O_n$-modules (also as a consequence of the inclusion $H_f\subseteq \Theta_{f, N}$):
\begin{equation}\label{isomHf}
\frac{\Theta_{h,M}+\Theta_{f,N}}{\Theta_{h,M}+H_f}\cong \frac{\Theta_{f,N}}{\Theta_{f,N}\cap(\Theta_{h,M}+H_f)}=
\frac{\Theta_{f,N}}{(\Theta_{f,N}\cap\Theta_{h,M})+H_f}.
\end{equation}

Relation (\ref{hMfK}) and the difference of relations (\ref{es2KM20}) and (\ref{es2KM2}) gives the equality
$$
\dim_\C\frac{\O_n}{J_{h, M}(f)}-\dim_\C\frac{\O_n}{J_{h, M}(f)+N}=\mu(f)-\dim_\C\frac{\O_n}{J(f)+N}+\dim_\C\frac{\Theta_{h,M}+\Theta_{f,N}}{\Theta_{h,M}+H_f}.
$$
Hence (\ref{numinustau2K}) follows by applying the isomorphism (\ref{isomHf}) to the above relation.
\end{proof}

In general, the complexity of the computation of modules of the form $\Theta_{h,M}$ is high. However,
the next result shows that the colength of the module $\Theta_{h,M}+\Theta_{f,N}$ in $\O_n^n$ admits an expression
as the sum of colengths of simpler modules.

\begin{cor}\label{codsumats}
Let $h:(\C^n,0)\to (\C^p,0)$ and $f:(\C^n,0)\to (\C^q,0)$ be analytic map germs and
let $M\subseteq \O_n^p$ and $N\subseteq \O_n^q$ be submodules. Let us suppose that $\t_{M\oplus N}(h,f)$ is finite.
Then the following relation holds:
\begin{equation}\label{thetahMfN}
\dim_\C\frac{\O_n^n}{\Theta_{h,M}+\Theta_{f,N}}=\t_{M\oplus N}(h,f)-\t_{M}(h)-\t_{N}(f).
\end{equation}
\end{cor}

\begin{proof} By Corollary \ref{diffsym}, the finiteness of $\t_{M\oplus N}(h,f)$ implies that
the modules $J_{h,M}(f)+N$ and $D(f)+N$ have finite colength in $\O_n^q$ and that the modules
$J_{f,N}(h)+M$ and $D(h)+M$ have finite colength in $\O_n^p$. By (\ref{ses2}) and (\ref{es2KM}) we have
$$
\t_{N\oplus M}(f,h)-\t_{M}(h)=\t_N(f)+\dim_\C\frac{\O_n^n}{\Theta_{h, M}+\Theta_{f, N}}.
$$
Hence (\ref{thetahMfN}) follows.
\end{proof}

As an immediate application of the previous result we obtain the following conclusion concerning the Tjurina number of an {\icis}.

\begin{cor}
Let $h:(\C^n,0)\to (\C^p,0)$ and $f:(\C^n,0)\to (\C^q,0)$ be analytic maps for which
$(f,h):(\C^n,0)\to (\C^{p+q},0)$ is an {\icis}, where $2\leq p+q\leq n$. Let $I$ denote the ideal of $\O_n$ generated by the component functions
of $(f,h)$. Then
\begin{equation}
\tau(f,h)=\dim_\C\frac{\O_n^n}{\Theta_{h,I}+\Theta_{f,I}}+\t_{I}(h)+\t_{I}(f).
\end{equation}
\end{cor}

\begin{proof}
Since we assume that $(f,h)$ is an {\icis}, we automatically have that $\t_{I}(f,h)$ is finite.
Hence the result follows by taking $M=I^{\oplus p}$ and $N=I^{\oplus q}$ in (\ref{thetahMfN}).
\end{proof}

In the remaining section we will show additional applications of Theorem \ref{claudetot}\,(\ref{ses2}) and Proposition \ref{prop-1}\,(\ref{es2KM}).

Let $X$ be an analytic subvariety of $(\C^n,0)$ and let $f\in\O_n$.
As indicated in Remark \ref{introdThetaX}\,\ref{defmuX}, we denote by $\mu_X(f)$ the Bruce-Roberts number
of $f$ with respect to $X$. In addition to $\mu_X(f)$ we also consider the following numbers, which also involve the module $\Theta_X$:
\begin{align*}
\mu_X^-(f)&=\dim_\C\frac{\O_n}{J_X(f)+I(X)}\\
\tau_X(f)&=\dim_\C\frac{\mathcal O_n}{J_X(f)+\langle f \rangle}\\
\tau_X^-(f)&=\dim_\C\frac{\O_n}{J_X(f)+\langle f \rangle+I(X)},
\end{align*}
whenever the lengths of the quotient rings appearing on the right of the above definitions are finite.
The numbers $\mu_X^-(f)$ and $\tau_X(f)$ are known as the
the {\it relative Bruce-Roberts number of $f$ with respect to $X$}
(see \cite{BR,LNOT4, LNOT5})
and the {\it Bruce-Roberts' Tjurina number of $f$ with respect to $X$} (see \cite{BiviaRuas}), respectively.
By analogy, we also refer to $\tau_X^-(f)$ as the {\it relative Bruce-Roberts' Tjurina number of $f$ with respect to $X$}, i.\,e.\, we add the term {\it relative} to the express that we compute the colengths in the quotient ring $\O_n/I(X)$.

The effective computation of these numbers is in general a non-trivial problem, as it requires knowledge of an explicit system of generators
of the module $\Theta_X$. The next result helps to simplify the computation of $\mu_X(f)$, $\tau_X(f)$ and $\mu^-_X(f)$ when $X$ is
an {\icis}, since it gives an expression for these numbers as the difference of colengths of two modules.
The resulting formulas allows to understand the nature of these numbers and to relate them with other invariants.

\begin{cor}\label{segonaexp}
Let $h:(\C^n,0)\to (\C^p,0)$ be an {\icis} with $n-p\geq 1$, let $I=\langle h_1,\dots, h_p\rangle$ and let $X=h^{-1}(0)$.
Given a function germ $f\in\O_n$, we have that $\mu_X(f)<\infty$ if and only if the submodule
$D(f,h)+(0\oplus I^{\oplus p})$ has finite colength in $\O_n^{p+1}$. Moreover, in this case we have
\begin{align}
\mu_X(f)&=\dim_\C\frac{\O_n^{p+1}}{D(f,h)+(0\oplus I^{\oplus p} )}-\tau(h)  \label{formuXf0} \\
\tau_X(f)&=\dim_\C\frac{\O_n^{p+1}}{D(f,h)+(\langle f\rangle\oplus I^{\oplus p} )}-\tau(h)  \label{tauXficis}\\
\mu_X^-(f)&=\dim_\C\frac{\O_n^{p+1}}{D(f,h)+I^{\oplus (p+1)}}-\tau(h)  \label{nuXficis}\\
\tau_X^-(f)&=\dim_\C\frac{\O_n^{p+1}}{D(f,h)+\big((\langle f \rangle+I)\oplus I^{\oplus p} \big)}-\tau(h).  \label{relnuXficis}
\end{align}

In particular, we obtain the following relations:
\begin{align}
\mu_X(f)-\tau_X(f)&=\dim_\C\frac{D(f,h)+(\langle f\rangle \oplus I^{\oplus p})}{D(f,h)+(0 \oplus I^{\oplus p})} \label{nmt1}  \\
\mu_X^-(f)-\tau_X^-(f)&=\dim_\C\frac{D(f,h)+\big((\langle f \rangle+I)\oplus I^{\oplus p} \big)}{D(f,h)+I^{\oplus (p+1)}}.  \label{nmt2}
\end{align}
\end{cor}

\begin{proof}
Let us remark that, since $h$ is an \textsc{icis} of positive dimension, the ideal $I$ is radical (see for
instance \cite[p.\,7]{Looijenga}). Therefore $J_X(f)=J_{h, I}(f)$, for any $f\in\O_n$.

Let us fix a function $f\in\O_n$. The exactness of (\ref{ses}) in the case
$M=I^{\oplus p}$ and $N=\{0\}$, shows that the finiteness of $\mu_X(f)$ is equivalent to the condition that the submodule
$D(f,h)+(0\oplus I^{\oplus p})$ has finite colength in $\O_n^{p+1}$.

Let us suppose that $\mu_X(f)<\infty$. Relation (\ref{formuXf0}) follows as a direct application of (\ref{ses2}) by taking
$M=I^{\oplus p}$ and $N=\{0\}$. Analogously, relations (\ref{tauXficis}), (\ref{nuXficis}) and (\ref{relnuXficis})
follow from (\ref{ses2}) by taking $M=I^{\oplus p}$ in the three cases and $N=\langle f\rangle$, $N=I$ and
$N=\langle f\rangle+I$, respectively.
Relations (\ref{nmt1}) and (\ref{nmt2}) are direct applications of (\ref{formuXf0})--(\ref{relnuXficis}).
\end{proof}

\begin{ex}
Let us consider the map $h:(\C^4,0)\to (\C^3,0)$ given by $h(x,y,z,t)=(x^2-y^2, xy+zt^2, x^3+y^2+z^3+t^2)$.
Let $X=h^{-1}(0)$ and let $I$ denote the ideal of $\O_3$ generated by the component functions of $h$.
We have that $h$ is an \textsc{icis}, $\dim(X)=1$, $\mu(h)=37$ and $\tau(h)=35$.
Let $f\in \O_4$ be given by $f(x,y,z,t)=x^2+y^2+z^2+t^2$. The module $D(f,h)+(0\oplus I^{\oplus 3})$ has finite colength and hence,
by Corollary \ref{segonaexp} we conclude that $\mu_X(f)=20$, $\tau_X(f)=16$, $\mu_X^-(f)=17$ and $\tau_X^-(f)=13$.
Let us remark that by using Singular \cite{Singular} and Lemma \ref{syzyg}, it is possible to find out that in this case $\Theta_X$ is an
intricate submodule of $\O_4^4$.
\end{ex}

\begin{rem}\label{sobresegonaexp}
\begin{enumerate}[label=(\alph*),wide]
\item As we will see in Section \ref{logideals}, the respective first summands appearing in (\ref{formuXf0})
and (\ref{nuXficis}) can be expressed in terms of other numerical invariants attached to the pair $(f,h)$. In order to
prove this fact, we develop in Section \ref{RSQM} a series of algebraic tools.

\item\label{mesgeneralencara} Let us remark that, as an application of (\ref{ses2}), relations (\ref{tauXficis}), (\ref{nuXficis}) and (\ref{relnuXficis}) also hold by assuming only the finiteness of $\tau_X(f)$, $\mu_X^-(f)$ and $\tau_X^-(f)$, respectively, instead of assuming that $\mu_X(f)<\infty$.
\end{enumerate}
\end{rem}

In the following result, which is similar to \cite[Proposition 2.1]{BKR1}, we derive some direct consequences of Proposition \ref{prop-1}\,(\ref{es2KM}).

\begin{cor}\label{applK}
Let $X$ be an analytic subvariety of $(\C^n,0)$ and let $I=I(X)$.
Let $f\in \O_n$ and let $Y=f^{-1}(0)$. Let us suppose that $\mu_X(f)$ is finite.
Then
\begin{align}
\mu_{X}(f)&
=\mu(f)+\dim_\C\frac{\O_n^n}{\Theta_X+H_f}  \label{muG}\\
\tau_X(f)&
=\tau(f)+\dim_\C\frac{\O_n^n}{\Theta_X+\Theta_Y}.\label{tauG}\\
\mu_X^-(f)&
=\dim_\C\frac{\O_n}{J(f)+I}+ \dim_\C\frac{\O_n^n}{\Theta_X+\Theta_{f, I}}\label{nubarG}\\
\tau_X^-(f)&=\dim_\C\frac{\O_n}{J(f)+\langle f\rangle+I}+ \dim_\C\frac{\O_n^n}{\Theta_X+\Theta_{f, I+\langle f\rangle}}.\label{taubarG}\\
\mu_X(f)-\tau_X(f)&=\mu(f)-\tau(f)+\dim_\C\frac{\Theta_Y}{H_f+(\Theta_X\cap \Theta_Y)}.\label{numinustau2}\\
\mu_X^-(f)-\tau_X^-(f)&=\dim_\C\frac{\O_n}{J(f)+I}-\dim_\C\frac{\O_n}{J(f)+\langle f\rangle+I}\label{numinustau22}\\
&\hspace{0.5cm}+
\dim_\C\frac{\Theta_{f, I+\langle f\rangle}}{\Theta_{f, I}+(\Theta_X\cap \Theta_{f, I+\langle f\rangle})}.\nonumber
\end{align}
\end{cor}

\begin{proof}
Relations (\ref{muG}), (\ref{tauG}), (\ref{nubarG}) and (\ref{taubarG}) arise as an application of equality (\ref{es2KM})
by taking the module $M$ as $I^{\oplus p}$ and the ideal $N$ as $\{0\}$, $\langle f\rangle$, $I$ or $\langle f\rangle+I$, respectively. Equality (\ref{numinustau2}) is a particular case of (\ref{numinustau2K}); this also follows by a direct argument by computing the difference of (\ref{muG}) and (\ref{tauG}). The proof of (\ref{numinustau22}) is analogous.
\end{proof}

Under the conditions of the above result, relation (\ref{numinustau2}) shows that $\mu_X(f)=\tau_X(f)$ if and only if
$\mu(f)=\tau(f)$ and $\Theta_Y=H_f+(\Theta_X\cap \Theta_Y)$, as already observed in  \cite{BKR1}. Let us remark that the condition
$\Theta_Y=H_f+(\Theta_X\cap \Theta_Y)$ links very significantly the modules $\Theta_X$ and $\Theta_Y$.

We end this section with a result concerning the computation of $\mu_X(f)$, $\tau_X(f)$, $\mu^-_X(f)$ and $\tau^-_X(f)$
when $X$ is a union of hypersurfaces. We first attach to any given analytic map another submodule
that will play an important role in this task.

 Given a map $h:(\C^n,0)\to (\C^p,0)$, let us define the following matrix $T^0(h)$:
\begin{equation}\label{matrixTg0}
T^0(h)=\left[
  \begin{array}{ccccccccccccc}
\frac{\partial h_1}{\partial x_1}&\cdots & \frac{\partial h_1}{\partial x_n}& h_1 & \cdots & 0 \\
\vdots &  & \vdots & \vdots & \ddots  & \vdots                                      \\
\frac{\partial h_p}{\partial x_1}&\cdots & \frac{\partial h_p}{\partial x_n}& 0 & \cdots & h_p
  \end{array}
\right].
\end{equation}
As usual, if there is no risk of confusion, we also denote by $T^0(h)$ the submodule of $\O_n^p$ generated by the columns of the above matrix.
We denote by $\tau^0(h)$ the colength of $T^0(h)$ in $\O_n^p$, when this colength is finite.
Let us remark that if $h$ is an \icis, then $\tau^0(h)\geq \tau(h)$, since $T^0(h)\subseteq T(h)$ (see (\ref{matrixTh})).
Obviously equality holds when $p=1$.

\begin{rem}
The finiteness of the colength of the module $T^0(h)\subseteq \O_n^p$ implies that each row of the matrix given in (\ref{matrixTg0})
must generate an ideal of finite colength in $\O_n$, which is to say that $h_i$ has isolated singularity at the origin, for all $i=1,\dots, p$.
We refer to \cite[Theorem 1.1]{BKS} for the equivalence stating that the ideal $\I_p(T^0(h))$ has finite colength in $\O_n$ if and only
if $\mu(h_i)<\infty$, for all $i=1,\dots, p$, and the varieties $h_1^{-1}(0),\dots, h_p^{-1}(0)$ are in general position away from the origin.
\end{rem}

\begin{cor}\label{unionXi}
Let $h=(h_1,\dots, h_p):(\C^n,0)\to (\C^p,0)$ denote an analytic map such that $\tau^0(h)<\infty$
and let $X_i=h_i^{-1}(0)$, for all $i=1,\dots, p$.
Let us suppose that $h_1,\dots, h_p$ are pairwise relatively prime.
Let $X=V(h_1)\cup\cdots \cup V(h_p)$ and let $f\in\O_n$ such that $\mu_X(f)<\infty$. Then
\begin{align}
\mu_X(f)&=\dim_\C\frac{\O_n^{p+1}}{D(f,h)+(0\oplus\langle h_1\rangle\oplus \cdots \oplus\langle h_p\rangle          )}-\tau^0(h)     \label{muXficis2}\\
\tau_X(f)&=\dim_\C\frac{\O_n^{p+1}}{D(f,h)+(\langle f\rangle\oplus\langle h_1\rangle\oplus \cdots \oplus\langle h_p\rangle)}-\tau^0(h)
=\tau^0(f,h)-\tau^0(h).  \label{tauXficis2}\\
\mu_X^-(f)&=\dim_\C\frac{\O_n^{p+1}}{D(f,h)+(\langle h_1\cdots h_p\rangle\oplus\langle h_1\rangle\oplus \cdots \oplus\langle h_p\rangle)}-\tau^0(h).  \label{nuXficis2}\\
\tau_X^-(f)&=\dim_\C\frac{\O_n^{p+1}}{D(f,h)+(\langle h_1\cdots h_p, f\rangle\oplus\langle h_1\rangle\oplus \cdots \oplus\langle h_p\rangle)}-\tau^0(h).  \label{tauXbaricis}
\end{align}
\end{cor}

\begin{proof}
Let $X_i= h_i^{-1}(0)$, for all $i=1,\dots, p$. The condition that $h_1,\dots, h_p$ are pairwise relatively prime imply that
$\Theta_X=\Theta_{X_1}\cap \cdots \cap \Theta_{X_p}$. In particular, $\Theta_X=\Theta_{h,M}$,
where $M=\langle h_1\rangle \oplus \cdots \oplus \langle h_p\rangle$. Hence the result arises as a direct application
of relation (\ref{ses2}).
\end{proof}

Let us fix an analytic function germ $f:(\C^n,0)\to (\C,0)$ with isolated singularity at the origin.
In view of the preceding result, we can consider the sequences of analytic invariants of $f$ given by
$\mu_{X^{(r)}}(f)$, $\tau_{X^{(r)}}(f)$, $\mu^-_{X^{(r)}}(f)$ and $\tau^-_{X^{(r)}}(f)$, where $X^{(r)}$ denotes the
union of $r$ generic hyperplanes in $\C^n$ passing through the origin.
Thus $X^{(r)}=h_1^{-1}(0)\cup\cdots\cup h_r^{-1}(0)$, where $(h_1,\dots, h_r):\C^n\to \C^r$ is a generic linear map.
Let us remark that, since $h_1,\dots, h_r$ are generic linear forms of $\C[x_1,\dots, x_n]$, we have
$\tau^0(h_1,\dots, h_r)=0$ provided that $r\leq n$.

We refer to \cite{Holm,RoseTerao,Wiens,Yuz} for references whose main objective is the effective computation of a minimal generating
system of $\Theta_X$ when $X$ is an arrangemenent of subspaces of $\C^n$. We also refer to \cite{OT,SS} for further references
about the study of several algebraic properties of $\Theta_X$ and its interplay with geometrical and combinatorial problems in this case.

\begin{ex}
Let $f,g\in\O_3$ be the function germs given by $f(x,y,z)=xyz+x^4+y^4+z^4$ and
$g(x,y,z)=xy^2z+x^5+y^4+z^5$, respectively. By applying Corollary \ref{unionXi}, the following values for
$\mu_{X^{(r)}}(f)$, $\tau_{X^{(r)}}(f)$, $\mu^-_{X^{(r)}}(f)$ and $\tau^-_{X^{(r)}}(f)$ are obtained when $1\leq r\leq 5$:
\vspace{0.1cm}

\centerline{
\begin{tabular}{l*{6}{c}r}
$r$              & $\mu_{X^{(r)}}(f)$ & $\tau_{X^{(r)}}(f)$ & $\mu^-_{X^{(r)}}(f)$ & $\tau^-_{X^{(r)}}(f)$ \\
\hline
1 & 15 & 13 & 4 & 4  \\
2 & 21 & 16 & 10 & 10 \\
3 & 30 & 22 & 19 & 18 \\
4 & 42 & 31 & 31 & 27\\
5 & 57 & 40 & 46 & 37 \\
\end{tabular}
\hspace{0.7cm}
\begin{tabular}{l*{6}{c}r}
$r$              & $\mu_{X^{(r)}}(g)$ & $\tau_{X^{(r)}}(g)$ & $\mu^-_{X^{(r)}}(g)$ & $\tau^-_{X^{(r)}}(g)$ \\
\hline
1 & 48 & 40 & 10 & 9  \\
2 & 61 & 47 & 23 & 20  \\
3 & 78 & 58 & 40 & 34  \\
4 & 99 & 70 & 61 & 50 \\
5 & 124 & 83 & 86& 67  \\
\end{tabular}
}
\end{ex}

\section{Regular sequences of quotient modules and $I$-consistency of maps}\label{RSQM}

The main results of this section are Propositions \ref{propTor}, \ref{consistent} and  \ref{lesTheta}, and Corollary \ref{expgen1}.
The first is a result of commutative algebra that will imply a characterization, in Corollary \ref{IM}, of regular sequences of quotient modules.
Let $I$ be an ideal of $\O_n$. In Corollary \ref{expgen1} we show an expression for the colength of an ideal of the form $J_{h, I}(f)+N$, where $N\subseteq I^{\oplus q}$ and $h:(\C^n,0)\to (\C^p,0)$ and $f:(\C^n,0)\to (\C^q,0)$ are analytic map germs.
Proposition \ref{consistent} has motivated us to introduce in Definition \ref{defI-consistency} the condition of $I$-consistency of maps. This condition will play a fundamental
role in Sections \ref{logideals} and \ref{ThetaXT}.

Let $R$ be a ring, let $\lambda_1,\dots, \lambda_d\in R$ and let $M$ be an $R$-module.
We recall that $\lambda_1,\dots, \lambda_d$ is called a {\it regular sequence with respect to $M$} when
$M\neq \langle \lambda_1,\dots, \lambda_d\rangle M$, $\lambda_1$ is a non-zerodivisor of $M$ and
$\lambda_i$ is a non-zerodivisor of $M/\langle \lambda_1,\dots, \lambda_{i-1}\rangle M$, for all $i=2,\dots, d$
(see for instance \cite{BrunsHerzog, Eisenbud, Matsumura}). When $M=R$ then we will also say that $\lambda_1,\dots, \lambda_d$
is a {\it regular sequence of $R$}.

If $I$ is an ideal of $R$ and $M$ is a submodule of the free module
$R^q$, for some $q\geq 1$, then it is obvious that $IM\subseteq M\cap I^{\oplus q}$.
As a consequence of Corollary \ref{IM} we will see that, if
$I=\langle \lambda_1,\dots, \lambda_d\rangle$, being $\lambda_1,\dots, \lambda_d$ a regular sequence with respect to $R$ and with respect to the quotient module $R^q/M$, then
$IM= M\cap I^{\oplus q}$. That is, in this case $M\cap I^{\oplus q}$ reaches the smallest possible size.

We remark that the following result is already known in the case $q=1$ (see for instance \cite[p.\,646]{Eisenbud} or \cite[p.\,380]{GP}).

\begin{prop}\label{propTor}
Let $R$ be a Noetherian local ring and let $M$ be an $R$-submodule of $R^q$. Let $I$ be an ideal of $R$. Then
$$
\Tor_1^R\left( \frac{R^q}{M},\frac{R}{I} \right)=\frac{M\cap I^{\oplus q}}{IM}.
$$
\end{prop}

\begin{proof}
Let us consider the exact sequence
\begin{equation}\label{Mexactw}
\xymatrix@C=0.2cm@R=1.5ex{
0 \ar[rr] &&  M \ar[rr]^-{\pi_1} && R^q \ar[rr]^-{\pi_2}   &&
\displaystyle\frac{R^q}{M}\ar[rr] && 0}
\end{equation}
where $\pi_1$ and $\pi_2$ denote the inclusion and the projection morphisms.
The long Tor-sequence associated to the above sequence (see \cite[\S 6.2]{Eisenbud} or \cite[Proposition 7.1.2]{GP}) is written as follows:
\begin{align}
&\xymatrix@C=0.2cm@R=1.5ex{
\cdots \ar[rr] && \Tor_2^R(\frac{R^q}{M},\frac RI) \ar[rr] &&
\Tor_1^R(M,\frac RI) \ar[rr] &&
\Tor_1^R(R^q,\frac RI) \ar[rr]  &&
\Tor_1^R\(\frac{R^q}{M},\frac RI\) \ar[rr] && \,
}\nonumber\\
&\xymatrix@C=0.2cm@R=1.5ex{
\hspace{3.6cm}
\ar[rr] &&
\Tor_0^R(M,\frac RI) \ar[rr]^-{\pi_1'} &&
\Tor_0^R(R^q,\frac RI) \ar[rr]^-{\pi_2'}  &&
\Tor_0^R\(\frac{R^q}{M},\frac RI\) \ar[rr] && 0
}\label{Tor2}
\end{align}
where the maps $\pi_1'$ and $\pi_2'$ of (\ref{Tor2}) are the canonical morphisms
induced from (\ref{Mexactw}) after tensoring with $\frac RI$.

For any pair of $R$-modules $M_1$ and $M_2$ we have $\Tor_0^R(M_1,M_2)=M_1\otimes_R M_2$ (see for instance \cite[p.\,380]{GP}).
Therefore
\begin{align}
\Tor_0^R\(M,\smallfrac RI\)&=M\otimes_R\smallfrac RI\cong \smallfrac{M}{IM}  \label{MIM}\\
\Tor_0^R\(R^q,\smallfrac RI\)&=R^q\otimes_R\smallfrac RI\cong \smallfrac{R^q}{IR^q}=\smallfrac{R^q}{I^{\oplus q}}  \label{RqIq} \\
\Tor_0^R\(\smallfrac{R^q}{M},\smallfrac RI\)&=\smallfrac{R^q}{M}\otimes_R\smallfrac RI\cong
\smallfrac{R^q}{I^{\oplus q}+M}.  \nonumber
\end{align}
In particular, considering the identifications of (\ref{MIM}) and (\ref{RqIq}),
we conclude that the morphism $\pi_1': \frac{M}{IM} \to \frac{R^q}{I^{\oplus q}}$ is given by $m+IM\mapsto m+I^{\oplus q}$, for any $m\in M$.
Since $R^q$ is a free module, we have $\Tor^R_1(R^q, \frac RI)=0$ (see for instance \cite[p.\,161,\,\textsection 6.2]{Eisenbud}).
Therefore, the exactness of the long Tor-sequence associated to (\ref{Mexactw}) implies that
$$
\Tor_1^R\(\frac{R^q}{M},\frac RI\)\cong\ker(\pi_1')=\frac{M\cap I^{\oplus q}}{IM}
$$
and hence the result follows.
\end{proof}

\begin{cor}\label{IM}
Let $R$ be a Noetherian local ring and let $M$ be an $R$-submodule of $R^q$. Let $\lambda_1,\dots, \lambda_d$ be a regular sequence of $R$
and let $I$ be the ideal of $R$ generated by $\lambda_1,\dots, \lambda_d$. Then the following conditions are equivalent:
\begin{enumerate}[label=\textnormal{(\alph*)}]
\item $\lambda_1,\dots, \lambda_d$ is a regular sequence with respect to $R^q/M$.
\item $M\cap I^{\oplus q}=IM$.
\end{enumerate}
\end{cor}

\begin{proof}
Since $\lambda_1,\dots, \lambda_d$ is a regular sequence with respect to $R$, the Koszul complex of $\lambda_1,\dots, \lambda_d$
constitutes a free resolution of $R/I$. By tensorizing this complex with $\frac{R^q}{M}$, we obtain that
condition (a) is equivalent to
$$
\Tor_1^R\left( \frac{R}{I},\frac{R^q}{M} \right)=0
$$
as can be seen, for instance, in \cite[Lemma\,\,21.5,\,p.\,162]{HIO} or \cite[Corollary 2,\,p.\,55]{Serre}.
Hence the result is an immediate consequence of Proposition \ref{propTor}.
\end{proof}

In general, given any analytic map germ $g:(\C^n,0)\to \C^s$ and an ideal $I\subseteq \O_n$,
the sum $I^{\oplus n}+H_g$ gives a substantial part of $\Theta_{g,I}$. The next result characterizes the equality $I^{\oplus n}+H_g=\Theta_{g,I}$.
In Section \ref{ThetaXT} we will develop a further study of the quotient module $\Theta_{g,I}/(I^{\oplus n}+H_g)$
(see Corollary \ref{ThThT}).

\begin{prop}\label{consistent}
Let $g:(\C^n,0)\to \C^s$ be an analytic map germ and let
$I$ be any ideal of $\O_n$. Then the following conditions are equivalent:
\begin{enumerate}[label=\textnormal{(\alph*)}]
\item\label{primera} $D(g)\cap I^{\oplus s}=ID(g)$.
\item\label{cons2} $\Theta_{g, I}=I^{\oplus n}+H_g$.
\end{enumerate}
Moreover, if we assume that $I$ is generated by a regular sequence $\lambda_1,\dots, \lambda_d$ of $\O_n$, we conclude that any of the above conditions is equivalent to:
\begin{enumerate}[label=\textnormal{(c)}]
\item \label{cons3} $\lambda_1,\dots, \lambda_d$ is a regular sequence with respect to $\O_n^{s}/D(g)$.
\end{enumerate}
\end{prop}

\begin{proof} Let us suppose that condition \ref{primera} holds. Let $\delta\in \Theta_{g, I}$.
Therefore
$
\delta(g)\in D(g)\cap I^{\oplus s}=ID(g)
$.
Hence, there exists some $\gamma\in I^{\oplus n}$
for which $\delta(g)=\gamma(g)$, that is, $\delta-\gamma\in H_g$. In particular
$\delta\in I^{\oplus n}+H_g$, and hence the inclusion
$\Theta_{g, I}\subseteq I^{\oplus n}+H_g$ follows. Since the reverse inclusion is obvious,
we obtain condition \ref{cons2}.

Let us suppose now that condition \ref{cons2} holds. Any element of $D(g)\cap I^{\oplus s}$ is of the form
$\delta(g)$, for some $\delta\in\Theta_{g, I}$. So, let us fix a derivation
$\delta\in \Theta_{g,I}$. By hypothesis, there exist derivations
$\gamma\in I^{\oplus n}$ and $\rho\in H_g$ for which $\delta=\gamma+\rho$. Hence
$\delta(g)=\gamma(g)\in ID(g)$. That is, the inclusion
$D(g)\cap I^{\oplus s}\subseteq ID(g)$ follows. Since the reverse inclusion is obvious,
we obtain condition \ref{primera}.

Under the condition that $I$ is generated by a regular sequence $\lambda_1,\dots, \lambda_d$ of $\O_n$, the equivalence between
\ref{primera} and \ref{cons3} is an immediate application of Corollary \ref{IM}.
\end{proof}

As we will see in Section \ref{ThetaXT}, the submodules of the form $I^{\oplus n}+H_g$ will play a special role in general
(see Definition \ref{defdeThetaT} and Corollary \ref{ThThT}).
Proposition \ref{consistent} motivates us to introduce the following definition.

\begin{defn}\label{defI-consistency}
Let $g:(\C^n,0)\to \C^s$ be an analytic map germ and let
$I$ be any ideal of $\O_n$. We say that $g$ is {\it $I$-consistent} when $D(g)\cap I^{\oplus s}=ID(g)$.
\end{defn}

As can be seen in Theorem \ref{JXfIf}, the condition of $I$-consistency is satisfied in a wide class of examples.

\begin{prop}\label{lesTheta}
Let $h:(\C^n,0)\to (\C^p,0)$ and $f:(\C^n,0)\to (\C^q,0)$ be analytic map germs and let
$I$ be any ideal of $\O_n$. Let us suppose that $(f,h)$ is $I$-consistent. Then
\begin{equation}\label{eq1}
J_{h, I}(f)\cap I^{\oplus q}=ID(f)
\hspace{1cm}\textit{and}\hspace{1cm}
J_{f, I}(h)\cap I^{\oplus p}=ID(h).
\end{equation}
\end{prop}

\begin{proof}
Let us prove the equality $J_{h, I}(f)\cap I^{\oplus q}=ID(f)$. The second equality of (\ref{eq1}) is a consequence of the first by replacing the roles of $f$ and $h$.
Any element of $ID(f)$ is of the form $\gamma(f)$, where $\gamma$ belongs to $I^{\oplus n}$. So the inclusion
$J_{h, I}(f)\cap I^{\oplus q}\supseteq ID(f)$ is obvious. Let us take an element $g\in J_{h, I}(f)\cap I^{\oplus q}$. Then, there exists a derivation
$\delta\in \Theta_{f, I}\cap \Theta_{h, I}$ for which $g=\delta(f)$. Since $(f,h)$ is $I$-consistent,
Proposition \ref{consistent} shows the existence of derivations $\gamma_1\in I^{\oplus n}$ and $\gamma_2\in H_f\cap H_h$ such that
$\delta=\gamma_1+\gamma_2$. In particular $g=\delta(f)=\gamma_1(f)\in ID(f)$. Hence the first equality of (\ref{eq1}) follows.
\end{proof}

In following result we characterize the condition of $I$-consistency of a given map $g:(\C^n,0)\to \C^s$ when $\O_n^{s}/D(g)$ is Cohen-Macaulay.

\begin{thm}\label{JXfIf}
Let $g=(g_1,\dots, g_s):(\C^n,0)\to \C^s$ be a analytic map germ, where $s\leq n$.
Let $d$ denote the dimension of the module $\O_n^{s}/D(g)$.
Let us suppose that $\O_n^{s}/D(g)$ is Cohen-Macaulay and $d\geq 1$.
Let $I$ be the ideal of $\O_n$ generated by a regular sequence $\lambda_1,\dots, \lambda_d$ of $\O_n$.
Then, the following conditions are equivalent:
\begin{enumerate}[label=\textnormal{(\alph*)}]
\item\label{item1} $g$ is $I$-consistent.
\item\label{item2} $\J(g)+I$ has finite colength.
\item\label{item3} $\lambda_1,\dots, \lambda_d$ is a regular sequence with respect to $\O_n^{s}/D(g)$.
\end{enumerate}
\end{thm}

\begin{proof}
First of all, let us observe that, since we assume that $\lambda_1,\dots, \lambda_d$ is a regular sequence of $\O_n$, Proposition \ref{consistent} shows the equivalence between items \ref{item1} and \ref{item3}.

Let $A$ denote the quotient module $\O_n^{s}/D(g)$. Let $\pi:\O^{s}_n\lto (\O_n/I)^{s}$ be the natural projection. We have the following isomorphisms:
\begin{equation}\label{N/IN}
\frac{A}{IA}\cong \frac{\O_n^{s}}{I^{\oplus s} +D(g)}\cong \frac{(\O_n/I)^{s}}{\pi\big(D(g)\big)}.
\end{equation}
Therefore
\begin{equation}\label{dimIN}
\dim \frac{A}{IA}= \dim \frac{(\O_n/I)^{s}}{\pi\big(D(g)\big)} =\dim\frac{\O_n}{\J(g)+I},
\end{equation}
where the last equality is a consequence of \cite[Proposition 20.7]{Eisenbud}, because $\J(g)$ denotes the ideal generated
by the maximal minors of the differential matrix of $g$.

Since we assume that $A$ is Cohen-Macaulay of dimension $d$, we have that $\lambda_1,\dots, \lambda_d$ is a regular sequence with respect to $A$
if and only if $A/IA$ has dimension $0$ (see for instance
\cite[Corollary B.8.3, p.\,416]{GLS} or \cite[Theorem 2.1.2, p.\,58]{BrunsHerzog}). In turn, by relation (\ref{dimIN}), this is equivalent
to saying that the ideal $\J(g)+I$ has finite colength. Therefore conditions
\ref{item2} and \ref{item3} are equivalent, and the result follows.
\end{proof}

\begin{ex}\label{exCMp}
Let us consider the functions $h,f\in\O_3$ given by $h(x,y,z)=xy+z^4$ and $f(x,y,z)=yz$.
We observe that the map $(f,h):(\C^3,0)\to (\C^2,0)$ is not an {\icis}, since the zero set of the ideal $\J(f,h)+\langle f,h\rangle\subseteq \O_3$
has dimension $1$.
However $\O_3^{2}/D(f,g)$ is Cohen-Macaulay of dimension $1$, as can be checked by applying Singular \cite{Singular}.
Thefore, $(f,h)$ is $I$-consistent, for any principal ideal $I$ for which $\J(f,h)+I$ has finite colength,
by Theorem \ref{JXfIf}. For instance, this happens when $I=\langle x^2+y^2\rangle$.
\end{ex}

\begin{ex}\label{exCMp2}
Let $h:(\C^3,0)\to (\C^2,0)$ be the map given by $h(x,y,z)=(xy, xz)$ and let $f\in\O_3$ be the polynomial given by
$f(x,y,z)=x^2+y^2+z^2$. The map $(f,h)$ is not an {\icis}, since the dimension of $(f,h)^{-1}(0)$ equals $1$. However $\O_3^3/D(f,h)$ is Cohen-Macaulay of dimension $2$.
Let $I$ be the ideal of $\O_3$ generated by the regular sequence $y^3+z^3,\, xy+z^2$ of $\O_3$. This ideal verifies that
$\J(f,h)+I$ has finite colength. Therefore $(f,h)$ is $I$-consistent, by Theorem \ref{JXfIf}.
\end{ex}

In the following result we analyse a special case of Theorem \ref{JXfIf} that we will apply in Sections
\ref{logideals} and \ref{ThetaXT}.

\begin{cor}\label{JXfIf2}
Let $h=(h_1,\dots, h_p):(\C^n,0)\to (\C^p,0)$ be a analytic map germ, where $n-p\geq 1$, and let $f\in\O_n$ such that $(f,h):(\C^n,0)\to (\C^{p+1},0)$ is an {\icis}. Let $I$ be the ideal of $\O_n$ generated by $h_1,\dots, h_p$.
Then, the following conditions are equivalent:
\begin{enumerate}[label=\textnormal{(\alph*)}]
\item\label{item12} $(f, h)$ is $I$-consistent.
\item\label{item22} $h$ is an {\icis}.
\item\label{item220} $\J(f,h)+\langle h_1,\dots, h_p\rangle$ has finite colength.
\item\label{item32} $h_1,\dots, h_p$ is a regular sequence with respect to $\O_n^{p+1}/D(f,h)$.
\end{enumerate}
Consequently, under any of the above conditions the following equalities hold:
\begin{equation}\label{igualtatprincipi2}
J_{X}(f)\cap I=IJ(f)
\hspace{2cm}
J_{f, I}(h)\cap I^{\oplus p}=ID(h).
\end{equation}
\end{cor}

\begin{proof}
The fact that $(f,h)$ is an {\icis} implies that $f, h_1,\dots, h_p$ is a regular sequence in $\O_n$. In particular,
$h_1,\dots, h_p$ is a regular sequence in $\O_n$. Moreover, $\O_n^{p+1}/D(f,h)$ is
a Cohen-Macaulay module of dimension $p$, by  \cite[Proposition 6.12, p.\,108]{Looijenga}.
Therefore, by Theorem \ref{JXfIf}, items \ref{item12}, \ref{item220} and \ref{item32} are equivalent.

Let us suppose condition \ref{item220}.
The inclusion $\J(f,h)\subseteq \J(h)$ implies that $\J(h)+\langle h_1,\dots, h_p\rangle$ has also finite colength. Thus,
$h$ is an {\icis}.

Conversely, let us suppose that $h$ is an {\icis}. Since we assume that $(f,h)$ is also an {\icis},
the Lê-Greuel formula implies that $\J(f,h)+\langle h_1,\dots, h_p\rangle$
has finite colength and this colength equals $\mu(h)+\mu(f,h)$ (see for instance \cite[Theorem\,3.7.1]{Le}).
Hence \ref{item22} and \ref{item220} are also equivalent and the result follows.

The equalities of (\ref{igualtatprincipi2}) follow as direct applications of Proposition \ref{lesTheta}.
\end{proof}

Let us remark that, by applying a different argument, the first equality of (\ref{igualtatprincipi2}) was also obtained in \cite[Theorem 2.3]{LNOT4} when both $f$ and $h$ are hypersurfaces for which $\mu_X(f)$ is finite, being $X=h^{-1}(0)$.

If $h:(\C^n,0)\to (\C^p,0)$ and $f:(\C^n,0)\to (\C^q,0)$ denote analytic map germs, then
Theorem \ref{claudetot}\,(\ref{ses2}) is a tool to compute the colength of $J_{h,M}(f)+N$, for any
pair of submodules $M\subseteq \O_n^p$ and $N\subseteq \O_n^q$, whenever such colength is finite. In the next result we will focus our attention
on the case where $M=I^{\oplus p}$ and $N\subseteq I^{\oplus q}$, for a given ideal $I$ of $\O_n$, under the condition that
$(f,h)$ is $I$-consistent. By Theorem \ref{claudetot}\,(\ref{ses2}) we know that
$$
\dim_\C\frac{\O_n^q}{J_{h, I}(f)+N}=t_{N\oplus I^{\oplus p}}(f,h)-t_I(h).
$$
In the next result we see how the above relation can be improved under the assumption that $(f,h)$ is $I$-consistent.

\begin{cor}\label{expgen1}
Let $h:(\C^n,0)\to (\C^p,0)$ and $f:(\C^n,0)\to (\C^q,0)$ be analytic map germs. Let $I$ be an ideal of $\O_n$
and let $N\subseteq I^{\oplus q}$ be a submodule such that $D(f,h)+N\oplus I^{\oplus p}$ has finite colength in $\O_n^{q+p}$.
If $(f,h)$ is $I$-consistent, then
\begin{equation}\label{for00}
\dim_\C\frac{\O_n^q}{J_{h, I}(f)+N}= \dim_\C\frac{I^{\oplus q}}{ID(f)+N}+\t_I(f,h)-\t_I(h).
\end{equation}
\end{cor}

\begin{proof}
We first notice that, since $D(f,h)+N\oplus I^{\oplus p}$ has finite colength in $\O_n^{q+p}$, the
exactness of (\ref{ses}) implies that the submodules $J_{h, I}(f)+N$ and $D(h)+I^{\oplus p}$ have finite colength in $\O_n^q$ and $\O_n^p$, respectively.
By applying Theorem \ref{claudetot}\,(\ref{ses2}), we deduce that
\begin{equation}\label{for1}
\dim_\C\frac{\O_n^q}{J_{h, I}(f)+I^{\oplus q}}=\t_I(f,h)-\t_I(h).
\end{equation}
Moreover, we have
\begin{equation}\label{for2}
\frac{J_{h, I}(f)+I^{\oplus q}}{J_{h, I}(f)+N}\cong  \frac{I^{\oplus q}}{J_{h, I}(f)\cap I^{\oplus q}+N}=
 \frac{I^{\oplus q}}{ID(f)+N},
\end{equation}
where the last equality is an application of Proposition \ref{lesTheta}, since $(f,h)$ is $I$-consistent.
Therefore, by (\ref{for1}) and (\ref{for2}) we conclude the following:
\begin{align*}
\dim_\C\frac{\O_n^q}{J_{h, I}(f)+N}&=\dim_\C\frac{\O_n^q}{J_{h, I}(f)+I^{\oplus q}}+\dim_\C\frac{J_{h, I}(f)+I^{\oplus q}}{J_{h, I}(f)+N}\\
&=\t_I(f,h)-\t_I(h)+\dim_\C\frac{I^{\oplus q}}{ID(f)+N}
\end{align*}
and thus (\ref{for00}) follows.
\end{proof}

We will see in Section \ref{ThetaXT} that the number $\t_I(h)$ of (\ref{for00})
is particularly related with the module $\Theta_{h, I}$ (see Corollary \ref{ThThT}).

\begin{ex}\label{exfxyz}
Let us consider the map $h:(\C^3,0)\to (\C^2,0)$ given by
$$
h(x,y)=\big(xy(x+y+z),\, yz(x+y-z)\big)
$$ and the function $f\in\O_3$
given by $f(x,y,z)=xyz+x^4+y^4+z^4$, for all $(x,y,z)\in\C^3$. We have that $(f,h)$ is not an {\icis} (the zero set of the ideal $\J(f,h)+\langle f,h\rangle\subseteq \O_3$ has dimension $1$) but $\O_3^3/D(f,h)$ is Cohen-Macaulay of dimension $2$
(as can be checked by using Singular \cite{Singular}). Let $I$ be the ideal of $\O_3$ generated by the regular sequence $x+y+z,\,x^2+y^2$ of $\O_3$.
Since the ideal $\J(f,h)+I$ has finite colength, we conclude that $(f,h)$ is $I$-consistent, by Theorem \ref{JXfIf}. We also find that $t_I(h)=7$
and $t_I(f,h)=12$. We observe that the submodule $D(f,h)+(0\oplus I^{\oplus 2})$ has finite colength in $\O_3^3$.
Therefore, by Corollary \ref{expgen1} (and thus without knowing an explicit generating system of $\Theta_{h ,I}$) we have
$$
\dim_\C\frac{\O_3}{J_{h, I}(f)+N}= \dim_\C\frac{I}{IJ(f)+N}+5
$$
for any ideal $N\subseteq I$. In particular
$$
\dim_\C\frac{\O_3}{J_{h, I}(f)}= \dim_\C\frac{I}{IJ(f)}+5=19.
$$
\end{ex}

\begin{rem}
The case of Corollary \ref{expgen1} where $h$ is an {\icis}, $I$ equals the ideal generated by the component functions of $h$ and $q=1$ will be analysed in the next section. As we will see in Corollary \ref{icislgth}, the term $t_I(f,h)$
of (\ref{for00}) will split as $\mu(f,h)+\mu(h)$ in this case, by virtue of the Lê-Greuel formula and a preliminary result about the
colength of submodules (see Lemma \ref{colparameterM}).
\end{rem}

Under the conditions of Corollary \ref{expgen1}, and assuming also that $q=1$, the next lemma will help us to compare $\mu(f)$ with the colength of $IJ(f)$ in $I$.
Given a Noetherian local ring $(R,\m)$, an $\m$-primary ideal $J$ of $R$ and an $R$-module $M$,
we denote by $e(J;M)$ the multiplicity of $J$ with respect to $M$ (see for instance \cite[Definition 11.1.5]{HS}
or  \cite[p.\,107]{Matsumura}).

\begin{lem}\label{IJf}
Let $(R,\m)$ be a Noetherian local Cohen-Macaulay ring of dimension $d$.
Let $J$ be an $\m$-primary ideal of $R$ generated by $d$ elements.
Let $I$ be an ideal of $R$ such that $\dim \frac RI< d$. Then
\begin{equation}\label{IJI}
\ell\(\frac{R}{J}\)\leq \ell\(\frac{I}{JI}\)
\end{equation}
and equality holds when $I$ is generated by a non-zerodivisor of $R$.
\end{lem}

\begin{proof}
Since we assume that $J$ is generated by $d$ elements, we can apply \cite[Proposition 11.1.10\,(1)]{HS} to deduce that
\begin{equation}\label{IJIprevi}
e(J;I)\leq \ell\(\frac{I}{JI}\),
\end{equation}
where $e(J;I)$ denotes the multiplicity of $J$ with respect to $I$, considering $I$ as an $R$-module.

The natural exact sequence $
\xymatrix@C=0.1cm@R=5ex{
0 \ar[rr] &&  I \ar[rr] &&
R  \ar[rr]  &&  \frac{R}{I}  \ar[rr] && 0
}
$ shows that
$$
e(J)=e(J;I)+e\(J;\frac{R}{I}\)
$$
by \cite[Theorem 14.6]{Matsumura} (see also \cite[Theorem 11.2.3]{HS}).
The condition $\dim \frac RI<d$ implies that $e(J;\frac{R}{I})=0$, by \cite[14.2]{Matsumura} (this fact also arises as a direct
application of the definition of $e(J;\frac{R}{I})$). Therefore $e(J)=e(J;I)$.
Moreover, since $R$ is Cohen-Macaulay and $J$ is generated by $d$ elements,
we have $e(J)=\ell\(R/J\)$ (see for instance \cite[Proposition 11.1.10 (2)]{HS}
or \cite[Theorem 17.11]{Matsumura}). Therefore, (\ref{IJIprevi}) implies (\ref{IJI}).

We observe that, given an element $h\in R$, if $h$ is a non-zerodivisor of $R$, then the morphism $R\to \langle h\rangle$ given
by multiplication by $h$ induces an isomorphism of $R$-modules $R/J\cong \langle h\rangle /J\langle h\rangle$ and therefore
$$
\ell\(\frac{R}{J}\)= \ell\(\frac{\langle h\rangle}{J\langle h\rangle}\).
$$
\end{proof}

As a direct consequence of Lemma \ref{IJf} we obtain that,
if $f:(\C^n,0)\to (\C,0)$ is a function with isolated singularity at the origin and $I$ is any proper ideal of $\O_n$, then
\begin{equation}\label{IJf2}
\dim_\C\frac{\O_n}{J(f)}\leq \dim_\C\frac{I}{J(f)I}
\end{equation}
and equality holds if $I$ is a non-zero principal ideal.


\section{Logarithmic ideals associated to {\icis}}\label{logideals} 

Let $h:(\C^n,0)\to (\C^p,0)$ be an {\icis} and let $X=h^{-1}(0)$.
In this section we show a direct application of Corollary \ref{expgen1} where we relate $\mu_X(f)$
 with other numerical invariants attached to $h$ and $f$, where $f:(\C^n,0)\to (\C,0)$ is an analytic function germ such that $\mu_X(f)$ is finite.
Hence, by applying a different argument, we deduce in Corollary \ref{noumuXf}\,(\ref{formuXf}) the same formula for $\mu_X(f)$ shown in \cite[p.\,45]{LNOT5}.

In the next result we will recall a fundamental fact of commutative algebra about  submodules of a free module.
Let $R$ be a Noetherian local ring and let $M\subseteq R^p$ be a submodule.
We denote by $\I_p(M)$ the ideal generated by the minors of order $p$
of any generating matrix of $M$ (that is, of any matrix whose columns form a generating system of $M$).
If $M$ has finite colength, then
$M$ is generated by at least $\dim(R)+p-1$ elements (see for instance
\cite[p.\ 213]{BRim}).
We denote by $e(M)$ the Buchsbaum-Rim multiplicity of $M$
(we refer to \cite{BRim},
\cite[\S 16]{HS} and \cite[\S 8]{Vasconcelos} for the definition
and properties of this notion, which generalizes the usual notion of multiplicity of ideals).
If $M$ admits a generating system formed by $\dim(R)+p-1$ elements and $M$ has finite colength, then we
say that $M$ is a {\it parameter submodule}.

\begin{lem}\label{colparameterM}
Let $R$ be a Noetherian Cohen-Macaulay local ring of dimension $d$. Let $M\subseteq R^{p+1}$ be a submodule
generated by $d$ elements, where $d-1\geq p\geq 1$.  Let $I$ be an ideal of $R$ such that
$\dim R/I=d-p$ and $I$ is generated by $p$ elements. Let us suppose that the colength of the ideal $\I_{p+1}(M)+I$ is finite. Then
\begin{equation}\label{param}
\ell\(  \frac{R}{\I_{p+1}(M)+I} \)=\ell\(  \frac{R^{p+1}}{M+I^{\oplus(p+1)}} \)=e(A),
\end{equation}
where $A$ denotes the image of $M$ in $(R/I)^{p+1}$.
\end{lem}

\begin{proof}
Let $S$ denote the quotient ring $R/I$.
We have the following isomorphisms:
\begin{align}
\frac{S^{p+1}}{A} &\cong \frac{R^{p+1}}{M +I^{\oplus (p+1)}} \\
\frac{S}{\I_{p+1}(A)} &\cong \frac{R}{\I_{p+1}(M)+I}.  \label{mod2}
\end{align}
By hypothesis the ideal $\I_{p+1}(M)+I$ has finite colength. Hence (\ref{mod2}) gives that
$\I_{p+1}(A)$ has finite colength. In particular, $A$ has finite colength too
(see \cite[p.\ 214]{Gaffney96} or \cite[Theorem 8.30]{Vasconcelos}). Moreover, $A$ is generated
by $d$ elements, being $d=p+1+\dim(S)-1$. Hence $A$ is a parameter submodule of $S^{p+1}$.
Therefore, by \cite[p.\ 214]{Gaffney96} (see also \cite[Theorem 8.30]{Vasconcelos}) we have that
$$
e(A)=\dim_\C\frac{S^{p+1}}{A}=\dim_\C\frac{S}{\I_{p+1}(A)},
$$
which shows (\ref{param}).
\end{proof}

\begin{cor}\label{icislgth}
Let $h=(h_1,\dots, h_p):(\C^n,0)\to (\C^p,0)$ be an {\icis} with $n-p\geq 1$ and
let $I$ be the ideal of $\O_n$ generated by $h_1,\dots, h_p$.
Let $f\in\O_n$ such that the map $(f,h):(\C^n,0)\to (\C^{p+1},0)$ is also an {\icis}.
Then
\begin{equation}\label{segonc}
\t_I(f,h)=\dim_\C\frac{\O_n^{p+1}}{D(f,h)+I^{\oplus(p+1)}}=\dim_{\C}\frac{\O_n}{\J(f,h)+I}=\mu(h)+\mu(f,h).
\end{equation}
\end{cor}

\begin{proof}
Since both $h$ and $(f,h)$ are {\icis}, the Lê-Greuel formula \cite{Le} gives that the ideal $\J(f,h)+I$ has finite colength in $\O_n$ and this colength is equal to $\mu(f,h)+\mu(h)$. Therefore (\ref{segonc}) is a direct application of Lemma \ref{colparameterM} and
the formula of Lê-Greuel.
\end{proof}

Let us remark that, as a direct consequence of Corollaries \ref{segonaexp}\,(\ref{nuXficis}) and \ref{icislgth}, we obtain the relation
\begin{equation}\label{lanubarra}
\mu_X^-(f)=\mu(f,h)+\mu(h)-\tau(h)
\end{equation}
for any $f\in \O_n$ for which $\mu_X(f)<\infty$, where $X=h^{-1}(0)$ and $h:(\C^n,0)\to (\C^p,0)$ is an {\icis} (see also \cite[Theorem 2.2]{LNOT5}). This conclusion will also follow as a consequence of the next result. By
Remark \ref{sobresegonaexp}\,\ref{mesgeneralencara} and Corollary \ref{icislgth}, relation (\ref{lanubarra}) remains true by only assuming
that $\mu_X^-(f)<\infty$ and $(f,h)$ is an {\icis}.

\begin{cor}\label{noumuXf}
Let $h:(\C^n,0)\to (\C^p,0)$ be an {\icis} with $n-p\geq 1$, let $I=\langle h_1,\dots, h_p\rangle$ and let $X=h^{-1}(0)$.
Let $f\in\O_n$ such that $\mu_X(f)<\infty$. Then $(f,h)$ is an {\icis} and
\begin{equation}\label{formuXf+N}
\dim_\C\frac{\O_n}{J_X(f)+N}=\dim_\C\frac{I}{IJ(f)+N}+\mu(f,h)+\mu(h)-\tau(h),
\end{equation}
for any ideal $N\subseteq I$. In particular, we obtain the following expression for $\mu_X(f)$:
\begin{equation}\label{formuXf}
\mu_X(f)=\dim_\C\frac{I}{IJ(f)}+\mu(f,h)+\mu(h)-\tau(h).
\end{equation}
\end{cor}

\begin{proof}
The condition $\mu_X(f)<\infty$ implies that the map $(f,h):(\C^n,0)\to (\C^{p+1},0)$ is an \textsc{icis}, by \cite[Proposition 2.8]{BiviaRuas}.
Hence $(f,h)$ is $I$-consistent, by Theorem \ref{JXfIf}. Therefore, by
Corollary \ref{expgen1} and the definition of Tjurina number of an {\icis} (see (\ref{tauicis})), we deduce that
\begin{align}
\dim_\C\frac{\O_n}{J_{h, I}(f)+N}&= \dim_\C\frac{I}{IJ(f)+N}+\t_I(f,h)-\tau(h)   \nonumber\\
&=\dim_\C\frac{I}{IJ(f)+N}+\mu(f,h)+\mu(h)-\tau(h)  \label{sumadenus}
\end{align}
for any ideal $N\subseteq I$, where we have applied Corollary \ref{icislgth} in (\ref{sumadenus}).
Hence (\ref{formuXf+N}) follows. The case $N=0$ of (\ref{formuXf+N}) leads to (\ref{formuXf}).
\end{proof}

\begin{rem}\label{remfinal}
\begin{enumerate}[label=(\alph*),wide]
\item Under the conditions of Corollary \ref{noumuXf}, let us denote the number
$\mu(f)+\mu(f,h)+\mu(h)-\tau(h)$ by $\kappa(f,h)$.
By (\ref{IJf2}) and (\ref{formuXf}) we obtain
$$
\mu_X(f)=\kappa(f,h)+\dim_\C \frac{I}{IJ(f)}-\mu(f)\geq \kappa(f,h)
$$
and equality holds when $p=1$. That is, the case $p=1$ of Corollary \ref{segonaexp} leads to the formula proven independently in
\cite[Theorem 1]{Kour} and \cite[Corollary 4.1]{LNOT3}.

\item Let us also remark that if we take $p=n-1$ in Corollary \ref{noumuXf}, then the map
$(f,h):(\C^n,0)\to (\C^n,0)$ becomes a zero-dimensional {\icis} whose Milnor number is given by
$\mu(f,h)=\dim_\C\O_n/(\langle f\rangle+I)-1$ (see \cite[p.\,261]{Greuel} or \cite[p.\,78]{Looijenga}).

\end{enumerate}
\end{rem}

We end this section by relating (\ref{formuXf}) with the $\mu^*$-sequence of a given function germ of $\O_n$.
Let $f:(\C^n,0)\to (\C,0)$ be a function germ with isolated singularity at the origin.
Let $\mu^*(f)=(\mu^{(n)}(f),\dots, \mu^{(1)}(f), \mu^{(0)}(f))$ be the $\mu^*$-sequence of $f$ defined by Teissier
in \cite[p.\,300]{Cargese}. Following the notation introduced in \cite[p.\,463]{BiviaRuas}, given an index $i\in\{1,\dots, n\}$, we denote by $\mu_{H^{(i)}}(f)$ the minimum value of $\mu_H(f)$ when $H$ varies in the set of linear subspaces of $\C^n$ of dimension $i$. That is, $\mu_{H^{(i)}}(f)$ is equal to the value of $\mu_H(f)$,
for a generic $i$-dimensional linear subspace $H\subseteq \C^n$. By (\ref{formuXf}), when $i\in\{0,1,\dots,n-1\}$, we have
\begin{align}
\mu_{H^{(i)}}(f)&=\dim_\C\frac{\langle h_1,\dots, h_{n-i}\rangle}{\langle h_1,\dots, h_{n-i}\rangle J(f)}+\mu(f, h_1,\dots, h_{n-i}) \nonumber\\
&=
\dim_\C\frac{\langle h_1,\dots, h_{n-i}\rangle}{\langle h_1,\dots, h_{n-i}\rangle J(f)}+\mu^{(i)}(f)  \label{elsgamma}
\end{align}
for a generic choice of linear forms $h_1,\dots, h_{n-i}\in\C[x_1,\dots, x_n]$. We recall that $\mu^{(0)}(f)=1$, by definition.

Given an index $i\in\{0,1,\dots,n-1\}$, relation (\ref{elsgamma}) motivates us to consider the common value of
$$
\dim_\C\frac{\langle h_1,\dots, h_{n-i}\rangle}{\langle h_1,\dots, h_{n-i}\rangle J(f)}
$$
when taking a generic linear map $(h_1,\dots, h_{n-i}):(\C^n,0)\to (\C^{n-i},0)$.
This is an analytic invariant of $f$, that we will denote by $\gamma^{(i)}(f)$.
Let $\m_n$ denote the maximal ideal of $\O_n$. We remark that
\begin{align*}
\gamma^{(0)}(f)&=\dim_\C\frac{\m_n}{\m_nJ(f)}=\dim_\C\frac{\O_n}{\m_nJ(f)}-1\\
\gamma^{(n-1)}(f)&=\mu(f)
\end{align*}
where the second relation follows as an application of Lemma \ref{IJf}. Hence $\mu_{H^{(n-1)}}(f)=\mu(f)+\mu^{(n-1)}(f)$, as already observed in \cite[p.\,464]{BiviaRuas}. It is worth to remark that, in turn, the number $\mu(f)+\mu^{(n-1)}(f)$ equals the multiplicity
of the Jacobian ideal of $f$ in the ring $\O_n/\langle f\rangle$ (see \cite[p.\,322]{Cargese}).

As a direct consequence of (\ref{elsgamma}) we obtain the following result, which
links the numbers $\ord(f)$, $\mu_{H^{(1)}}(f)$ and $\gamma^{(1)}(f)$.

\begin{cor}
Let $f\in\O_n$ with isolated singularity at the origin. Then
\begin{equation}\label{BRandorder}
\mu_{H^{(1)}}(f)=\gamma^{(1)}(f)+\ord(f)-1.
\end{equation}
\end{cor}

\section{The computation of $\t_I(h)$ in terms of the module $\Theta_{h,I}$}\label{ThetaXT}

Let $h=(h_1,\dots, h_p):(\C^n,0)\to (\C^p,0)$ be an {\icis} with $n-p\geq 1$. Let $X= h^{-1}(0)$. Let us consider the matrix of size $(p+1)\times n$ given by
\begin{equation}\label{Ipmes1}
\left[
            \begin{array}{ccc}
              \frac{\partial}{\partial x_1} & \cdots & \frac{\partial}{\partial x_n} \\
              \frac{\partial h_1}{\partial x_1} & \cdots & \frac{\partial h_1}{\partial x_n} \\
              \vdots & \, & \vdots \\
              \frac{\partial h_p}{\partial x_1} & \cdots & \frac{\partial h_p}{\partial x_n} \\
            \end{array}
          \right]
\end{equation}
where in the first row we have included the differential operators $\frac{\partial}{\partial x_1} ,\dots,\frac{\partial}{\partial x_n}$.
As remarked in \cite[1.5]{Vosegaard}, it is known that $H_h$ is generated by
the minors of order $p+1$ of the above matrix (see also \cite[p.\,231]{GLS}).
Let $J$ denote the ideal of $\O_n$ generated by $h_1,\dots, h_p$. We remark that $J^{\oplus n}+H_h\subseteq \Theta_X$. Also by \cite[1.5]{Vosegaard}, the submodule $H_h\subseteq \O_n^n$ is not an invariant of the ideal $J$ (that is, it depends on the fixed generating system of $J$), but the image of $H_h$ in $\Theta_X/J^{\oplus n}$
does not depend on $h$ whenever the components of $h$ generate $J$. In particular, the quotient module
$\overline\Theta_X$ given by
$$
\overline\Theta_X=\frac{\Theta_X}{H_h+J^{\oplus n}}
$$
only depends on the ideal $J$ and it is known that $\overline\Theta_X$ is a $\C$-vector space of dimension equal to the Tjurina number $\tau(h)$ (see \cite[3.6]{Vosegaard}, \cite[9.6]{Looijenga} or \cite[Proposition 4.6]{LNOT5}). We will also deduce that $\tau(h)=\dim_\C\overline\Theta_X$ in Corollary \ref{isomicis}, as a consequence of a series of results
developed in a more general context.

Let $I$ be an ideal of $\O_n$. In this section we characterize the number $\t_I(h)$, under some conditions,
as a measure of how far is $I^{\oplus n}+H_h$ from $\Theta_{h,I}$ (see also Proposition \ref{consistent}). In order to simplify the notation, we introduce the following definition.

\begin{defn}\label{defdeThetaT}
Let $h:(\C^n,0)\to (\C^p,0)$ be an analytic map germ and let $I$ be an ideal of $\O_n$.
We denote by $\Theta^\T_{h,I}$ the submodule of $\Theta_{h, I}$ given by $I^{\oplus n}+ H_h$.
We refer to $\Theta^\T_{h,I}$ as the submodule of {\it trivial vector fields} of $\Theta_{h,I}$.
Let us remark that it is immediate to prove that $\Theta^\T_{h,I}=\{\delta\in \O^n_n: \delta(h)\in ID(h)\}$.
If $f:(\C^n,0)\to (\C^q,0)$ denotes another analytic map germ, we denote by $J^\T_{h,I}(f)$ the submodule $J_{h,I}(f)$ given by
$$
J^\T_{h,I}(f)=\big\{\delta(f):\delta\in \Theta^\T_{h,I}\big\}.
$$

Analogous to Remark \ref{introdThetaX}\,\ref{adeintrodThetaX}, if $X$ denotes an analytic subvariety of $(\C^n,0)$ and
$I(X)$ is generated by the components of a given analytic map $h:(\C^n,0)\to (\C^p,0)$, then we denote
$\Theta^\T_{h,I(X)}$ also by $\Theta^\T_X$. Therefore, we also define the submodule
$J^\T_{X}(f)=\big\{\delta(f):\delta\in \Theta^\T_{X}\big\}$.
\end{defn}

In the following result we apply the notion of $I$-consistency to show an isomorphism of modules.
In Corollary \ref{isomicis} we will analyse the case where $h$ is an {\icis} of positive dimension and $I$ is the ideal of $\O_n$ generated by the component functions of $h$.

\begin{prop}\label{lesTheta2}
Let $h:(\C^n,0)\to (\C^p,0)$ and $f:(\C^n,0)\to (\C^q,0)$ be analytic map germs and let
$I$ be any ideal of $\O_n$. Let us suppose that $(f,h)$ is $I$-consistent.
The evaluation morphism $\varepsilon:\Theta_{h,I}\to J_{h,I}(f)$, given by
$\varepsilon(\delta)=\delta(f)$, for any
$\delta \in\Theta_{h,I}$, induces the isomorphisms of $\O_n$-modules
\begin{equation}\label{iso2}
\frac{\Theta_{h,I}}{\Theta^\T_{h,I}}\cong \frac{J_{h,I}(f)}{J^\T_{h,I}(f)}
\cong \frac{J_{h,I}(f)+I^{\oplus q}}{J^\T_{h,I}(f)+I^{\oplus q}}.
\end{equation}
\end{prop}

\begin{proof}
The evaluation morphism $\varepsilon:\Theta_{h,I}\to J_{h,I}(f)$ induces an epimorphism
$$
\varepsilon_1:\Theta_{h,I}\lto  \frac{J_{h,I}(f)}{J^\T_{h,I}(f)}.
$$
It is immediate to check that
\begin{equation}\label{kerovvarep}
\Theta_{h, I}^\T\subseteq \ker(\varepsilon_1)=I^{\oplus n}+H_h+H_f\cap \Theta_{h, I}.
\end{equation}

By Proposition \ref{consistent} we have $\Theta_{f, I}\cap \Theta_{h, I}=I^{\oplus n}+H_f\cap  H_h$. Therefore
$$
I^{\oplus n}+H_h+H_f\cap \Theta_{h, I}\subseteq
I^{\oplus n}+H_h+\Theta_{f, I}\cap \Theta_{h, I}=
I^{\oplus n}+H_h=\Theta_{h, I}^\T.
$$
Hence $\ker(\varepsilon_1)=\Theta_{h, I}^\T$ and we obtain the first isomorphism of (\ref{iso2}).

Analogously, the evaluation morphism $\varepsilon:\Theta_{h,I}\to J_{h,I}(f)$ also induces an epimorphism
$$
\varepsilon_2:\Theta_{h,I}\lto  \frac{J_{h,I}(f)+I^{\oplus q}}{J^\T_{h,I}(f)+I^{\oplus q}}
$$
whose kernel equals $I^{\oplus n}+H_h+\Theta_{f, I}\cap \Theta_{h, I}$.
By applying again Proposition \ref{consistent}, we obtain that $\ker(\varepsilon_2)=\Theta^\T_{h, I}$.
Hence the second isomorphism of (\ref{iso2}) arises.
\end{proof}

\begin{ex}\label{exIcons1}
Let us consider the function germs $f,h\in\O_3$ given by $f(x,y,z)=xyz$ and $h(x,y,z)=xyz+x^4+y^4+z^4$, for all $(x,y,z)\in\C^3$. Let $X=h^{-1}(0)$
and let $I$ denote the ideal of $\O_3$ generated by $h$.
We have that $\mu_X(f)$ is not finite (although $\mu_X^-(f)=50$), since $f$ has not isolated singularity at the origin. However
$(f,h):(\C^3,0)\to (\C^2,0)$ is an {\icis}. Therefore $(f,h)$ is $I$-consistent, that is, $D(f,h)\cap I^{\oplus 2}=ID(f,h)$, by Corollary \ref{JXfIf2}
(this can be also checked directly by using Singular \cite{Singular}). In particular, Proposition \ref{lesTheta2} shows us that $\Theta_X/\Theta_X^\T\cong J_X(f)/J_X^\T(f)$. Let us remark that, by Corollary \ref{JXfIf2},
$(f,h)$ is not $\langle f\rangle$-consistent, since $f$ has not isolated singularity at the origin.
\end{ex}

\begin{ex}\label{exIcons2}
Let $h$ be the function germ of Example \ref{exIcons1} and let $I=\langle x+y+z\rangle\subseteq\O_3$.
Hence $\Theta_{h, I}/\Theta_{h,I}^\T\cong J_{h, I}(f)/J_{h, I}^\T(f)$, for any function $f\in\O_3$ for which $(f,h)$ is $I$-consistent.
We find that the function $f$ of Example \ref{exIcons1} verifies that $(f,h)$ is $I$-consistent, since $\J(f,h)+I$ has finite colength
(see Theorem \ref{JXfIf}).
\end{ex}

The remaining section is devoted to compute the length of the quotient  module $\Theta_{h,I}/\Theta^\T_{h,I}$.

\begin{cor}\label{ThThT}
Let $h:(\C^n,0)\to (\C^p,0)$ be an {\icis} with $n-p\geq 1$. Let $I$ be an ideal of $\O_n$ generated by a regular sequence $\lambda_1,\dots, \lambda_p$ of $\O_n$. Let us suppose that the module $D(h)+I^{\oplus p}$ has finite colength in $\O_n^p$. Then
\begin{equation}\label{concl2}
\dim_\C \frac{J_{h,I}(f)+I}{J^\T_{h,I}(f)+I}=t_I(h)
\end{equation}
for any function germ $f\in\O_n$ such that $\J(f,h)+I$ has finite colength. If, moreover, we assume that $(f,h)$ is $I$-consistent, then
\begin{equation}\label{concl22}
\dim_\C\frac{\Theta_{h, I}}{\Theta^\T_{h, I}}=\dim_\C\frac{J_{h, I}(f)}{J^\T_{h, I}(f)}
=\dim_\C \frac{J_{h,I}(f)+I}{J^\T_{h,I}(f)+I}=t_I(h).
\end{equation}
\end{cor}

\begin{proof}
Let $f\in\O_n$ such that $\J(f,h)+I$ has finite colength.
Since $h$ is an {\icis} with $n-p\geq 1$, the module $H_h$ is generated by
the minors of order $p+1$ of the matrix given in (\ref{Ipmes1}) (see \cite[p.\,231]{GLS} or \cite[1.5]{Vosegaard}).
Thus $J^\T_{h,I}(f)=\J(f,h)+IJ(f)$. Hence, the following equalities hold:
\begin{equation}\label{parametermdl}
\dim_\C\frac{\O_n}{J^\T_{h,I}(f)+I}=\dim_\C\frac{\O_n}{\J(f,h)+I}=\dim_\C\frac{\O_n^{p+1}}{D(f,h)+I^{\oplus(p+1)}},
\end{equation}
where the second equality comes as a direct application of Lemma \ref{colparameterM}, that is,
from the fact that the image of $D(f,h)$ in $(\O_n/I)^p$ is a parameter submodule of $(\O_n/I)^p$.
Let us observe that
\begin{align}
\dim_\C\frac{J_{h,I}(f)+I}{J^\T_{h,I}(f)+I}&=\dim_\C\frac{\O_n}{J^\T_{h,I}(f)+I}-\dim_\C\frac{\O_n}{J_{h,I}(f)+I} \nonumber\\
&=\dim_\C\frac{\O_n}{J^\T_{h,I}(f)+I}-\(\dim_\C\frac{\O_n^{p+1}}{D(f,h)+I^{\oplus(p+1)}}-\dim_\C\frac{\O_n^p}{D(h)+I^{\oplus p}}\)  \label{thescnd}  \\
&=\dim_\C\frac{\O_n^p}{D(h)+I^{\oplus p}}\label{thethrd}
\end{align}
where (\ref{thescnd}) is a direct application of (\ref{ses2}) and we have applied relation (\ref{parametermdl}) to obtain (\ref{thethrd}).
That is, we conclude
\begin{equation}\label{concl1}
\dim_\C\frac{J_{h,I}(f)+I}{J^\T_{h,I}(f)+I}=\dim_\C\frac{\O_n^p}{D(h)+I^{\oplus p}}=t_I(h).
\end{equation}
Thus relation (\ref{concl2}) follows.
If, in addition, we assume that $(f,h)$ is $I$-consistent. Then we can join (\ref{concl1}) with the isomorphisms provided by
Proposition \ref{lesTheta2} to obtain (\ref{concl22}).
\end{proof}

\begin{ex}
Let us consider the functions $h,f\in\O_3$ given in Example \ref{exCMp}. That is, $f=yz$ and $h=xy+z^4$.
As explained in Example \ref{exCMp}, $(f,h)$ is not an {\icis} but $(f,h)$ is $I$-consistent,
for any  proper principal ideal $I$ of $\O_3$ for which $\J(f,h)+I$ has finite colength.
By Corollary \ref{ThThT}, relation (\ref{concl22}) holds for such ideals $I$.
For instance, if $I=\langle x^2+y^2\rangle$, we have
\begin{equation}\label{concl22ex}
\dim_\C\frac{\Theta_{h, I}}{\Theta^\T_{h, I}}=\dim_\C\frac{J_{h, I}(f)}{J^\T_{h, I}(f)}
=\dim_\C \frac{J_{h,I}(f)+I}{J^\T_{h,I}(f)+I}=t_I(h)=3.
\end{equation}
Let us remark that, if we take $I=\langle h\rangle$, then the ideal $\J(f,h)+I$ has not finite colength. Moreover
$J_X(f)=J_X^\T(f)$, so equality (\ref{concl2}) does not hold in this case, since $\tau(h)=\mu(h)=3$.
\end{ex}

As an immediate consequence of Corollary \ref{JXfIf2} and Corollary \ref{ThThT},
we obtain the following characterization of Tjurina numbers of {\icis}.

\begin{cor}\label{isomicis}
Let $h:(\C^n,0)\to (\C^p,0)$ be an {\icis} with $n-p\geq 1$, let $I=\langle h_1,\dots, h_p\rangle$ and let $X=h^{-1}(0)$.
Then
\begin{equation}\label{thtjur}
\dim_\C\frac{\Theta_X}{\Theta^\T_X}=\dim_\C\frac{J_X(f)}{J^\T_X(f)}
=\dim_\C \frac{J_X(f)+I}{J^\T_X(f)+I}=\tau(h)
\end{equation}
for all $f\in\O_n$ such that $(f,h):(\C^n,0)\to (\C^{p+1},0)$ is an {\icis}.
\end{cor}

\begin{proof}
Let $f\in \O_n$ for which the map $(f,h):(\C^n,0)\to (\C^{p+1},0)$ is an {\icis} (for instance, the function $f$ can be taken as a
generic linear form). Since we assume that $h$ is an {\icis}, then $(f,h)$ is $I$-consistent and
$\J(f,h)+I$ has finite colength, by Corollary \ref{JXfIf2}. Hence (\ref{thtjur}) follows from Corollary \ref{ThThT}.
\end{proof}

\begin{ex}
Let us consider the setup of Example \ref{exIcons1}. Although $\mu_X(f)$ is not finite, the map $(f,h):(\C^3,0)\to (\C^2,0)$ is an {\icis}.
So we can apply Corollary \ref{isomicis} to obtain that
$$
\dim_\C\frac{\Theta_X}{\Theta^\T_X}=\dim_\C\frac{J_X(f)}{J^\T_X(f)}
=\dim_\C \frac{J_X(f)+I}{J^\T_X(f)+I}=\tau(h)=10.
$$
\end{ex}

\vspace{3.1cm}

\noindent{\bf{Acknowledgements}}\\
Part of this work was developed during the visit of the first and the second authors to the Instituto de Ciências Matemáticas e de Computação, Universidade de São Paulo (campus de São Carlos, Brazil), in February-March 2023. They want to thank this institution for their hospitality and financial support. \\

\vspace{-0.2cm}

\noindent The authors dedicate this article to the memory of our colleague and friend Jim Damon.\\

\vspace{-0.2cm}

\noindent{\bf{Funding}}\\
The first author was partially supported by Grant PID2021-124577NB-I00 funded by
Ministerio de Ciencia, Innovación y Universidades MCIN/AEI/10.13039/501100011033 and by \lq ERDF A way of making Europe' (European Regional Development Fund).
The second author was partially supported by the UK Engineering and Physical Sciences Research Council, Grant EP/Y020669/1.
The third author was partially supported by Fundação de Amparo à Pesquisa do Estado de São Paulo, Proc. 2019/21181-0
and Conselho Nacional de Desenvolvimento Científico e Tecnológico, Proc. 305695/2019-3.

\end{document}